\documentclass{tac}
\usepackage{amsmath,amssymb,url,tikz,enumerate,rotating,color,hyperref}

\usetikzlibrary{matrix,decorations.markings}
\newcommand\sxto[1]{\mathbin{\smash{
\begin{tikzpicture}[baseline={([yshift=-1pt]current bounding box.south)}]
    \node (A) at (0,0) [inner xsep=0pt, inner ysep=1pt, minimum width=0.15cm] {\ensuremath{\scriptstyle #1}};
    \draw [->, line width=0.4pt, line cap=round]
        ([xshift=-2.5pt] A.south west)
        to ([xshift=3pt] A.south east);
\end{tikzpicture}}}}
\newcommand\sxfrom[1]{\mathbin{\smash{
\begin{tikzpicture}[baseline={([yshift=-1pt]current bounding box.south)}]
    \node (A) at (0,0) [inner xsep=0pt, inner ysep=1pt, minimum width=0.15cm] {\ensuremath{\scriptstyle #1}};
    \draw [<-, line width=0.4pt, line cap=round]
        ([xshift=-2.5pt] A.south west)
        to ([xshift=3pt] A.south east);
\end{tikzpicture}}}}
\tikzset{dagger/.style={
    decoration={markings, mark=at position 0.5 with {
        \draw [-] ++ (0,-.75mm) -- (0,.75mm);}
    }, postaction={decorate},
}}

\newcommand{\cat}[1]{\ensuremath{\mathbf{#1}}}
\newcommand{\op}{\ensuremath{{}^{\text{op}}}}
\newcommand{\id}[1][]{\ensuremath{\mathrm{id}_{#1}}}
\newcommand{\ie}{\textit{i.e.}\ }
\newcommand{\eg}{\textit{e.g.}\ }
\newcommand{\N}{\mathbb{N}}
\newcommand{\abs}[1]{\ensuremath{\left\vert {#1}\right\vert}}
\newcommand{\norm}[1]{\ensuremath{\left\Vert {#1}\right\Vert}}
\DeclareMathOperator{\coker}{coker}
\DeclareMathOperator{\dcolim}{dcolim} 
\DeclareMathOperator{\dlim}{dlim} 
\DeclareMathOperator{\im}{im}
\newcommand{\inprod}[2]{\ensuremath{\langle #1 \,,\, #2 \rangle}}

\title{Limits in dagger categories}
\author{Chris Heunen and Martti Karvonen}
\address{School of Informatics, University of Edinburgh\\10 Crichton Street, Edinburgh EH8 9AB, United Kingdom}
\eaddress{\{chris.heunen,martti.karvonen\}@ed.ac.uk}
\copyrightyear{2018}
\thanks{
We thank David Jordan, Tom Leinster, Peter Selinger, and the anonymous referee, for many heplful suggestions.
This work was supported by the Osk. Huttunen Foundation and EPSRC Fellowship EP/L002388/1. 
}

\begin{document}
\maketitle
\begin{abstract}
  We develop a notion of limit for dagger categories, that we show is suitable in the following ways:
  it subsumes special cases known from the literature;
  dagger limits are unique up to unitary isomorphism;
  a wide class of dagger limits can be built from a small selection of them;
  dagger limits of a fixed shape can be phrased as dagger adjoints to a diagonal functor;
  dagger limits can be built from ordinary limits in the presence of polar decomposition;
  dagger limits commute with dagger colimits in many cases.
\end{abstract}

\section{Introduction}

Dagger categories are categories with a functorial correspondence $\hom(A,B) \simeq \hom(B,A)$. They occur naturally in many settings, such as:
\begin{itemize}
  \item Much of algebra is based on structure-preserving functions, but sometimes many-valued homomorphisms are the right tool. More generally, \emph{relations} are useful in \eg algebraic topology and are fruitful to study abstractly~\cite{maclane:additiverelations}. Categories of relations can be axiomatised as certain dagger categories~\cite{puppe:korrespondenzen} and embed all exact categories~\cite{brinkmann:relations}. There is a large amount of literature on categories of relations~\cite{carboniwalters:cartesianbicategories,freydscedrov:catsalligators}. 

  \item There are many important \emph{involutive structures} in algebra. For example, a group, or more generally, an inverse semigroup~\cite{lawson:inversesemigroups}, may be regarded as a one-object dagger category that satisfies extra properties. So-called inverse categories~\cite{cockettlack:restrictioncategories} are used in computer science to model reversible computations~\cite{kaarsgaardaxelsengluck:joininversecategories}.

  \item	Just like monoids can be defined in a monoidal category, dagger categories provide the right infrastructure to study various internal involutive structures~\cite{egger:involutive}. Especially in \emph{quantum physics}, categorifications of involutive algebraic objects~\cite{ghezlimaroberts:wstarcategories} are used to model time-reversal~\cite{baez:quandaries,selinger:completelypositive}.
\end{itemize}
However, all of the above examples assume extra structure or properties on top of the ability to reverse morphisms. We wish to study dagger categories in their own right. 
Such a study is worthwhile because there is more to dagger category theory than formal category theory applied to categories with an extra structure (the dagger), for several reasons:
\begin{itemize}
  \item Dagger categories are \emph{self-dual} in a strong sense. Consequently dagger categories behave differently than categories on a fundamental level. Categories are intuitively built from objects $\bullet$ and morphisms $\bullet \rightarrow \bullet$. Dagger categories are built from $\bullet$ and $\bullet\leftrightarrows\bullet$. The result is a theory that is essentially directionless.
  For example, a dagger category has $\cat{J}$-shaped limits if and only if it has $\cat{J}\op$-shaped colimits (see also Section~\ref{sec:commutativity} below).
  Similarly, any dagger-preserving adjunction between dagger categories is ambidextrous~\cite{heunenkarvonen:daggermonads}.

  \item
  Objects in a dagger category do not behave the same when they are merely isomorphic, but only when the isomorphism respects the dagger.
  A case in point is the dagger category of Hilbert spaces and continuous linear functions. Two objects are isomorphic when there is a linear homeomorphism with respect to the two topologies induced by the inner products. The inner product itself need not be respected by the isomorphism, unless it is \emph{unitary}.
  The working philosophy for dagger categories is the `way of the dagger': all structure in sight should cooperate with the dagger.

  \item 
  There is no known way to translate categorical notions to their dagger categorical counterparts. There is a forgetful functor from the 2-category of dagger categories with dagger-preserving functors and natural transformations, to the 2-category of categories, but this forgetful functor has no 2-adjoints. 
  More seriously, we do want to consider non-dagger categories, for example when considering equalizers.
\end{itemize}

This article studies limits in dagger categories. This goal brings to the forefront the above two features of self-duality and unitarity. If $l_A \colon L \to D(A)$ is a limit for a diagram $D\colon\cat{J}\to\cat{C}$, then $l_A^\dag \colon D(A) \to L$ is a colimit for $\dag\circ D\colon\cat{J}\op\to\cat{C}$, so $L$ has two universal properties; it stands to reason that they should be compatible with each other. Moreover, a \emph{dagger limit} should be unique not just up to mere isomorphism but up to unitary isomorphism.

After Section~\ref{sec:dagcats} sets the scene, we define in Section~\ref{sec:daglims} the notion of dagger limit that subsumes all known examples. Section~\ref{sec:unitarity} shows how dagger limits are unique up to unitary isomorphism. 
Section~\ref{sec:completeness} deals with completeness. If a dagger category has `too many' dagger limits, it degenerates (showcasing how dagger category theory can be quite different than ordinary category theory). A more useful notion of `dagger completeness' is defined, and shown to be equivalent to having dagger equalizers, dagger products, and dagger intersections. 
Section~\ref{sec:globaldaglims} formulates dagger limits in terms of an adjoint to a diagonal functor, and Section~\ref{sec:daft} attempts a dagger version of an adjoint functor theorem. Section~\ref{sec:polar} makes precise the idea that polar decomposition turns ordinary limits into dagger limits. Finally, Section~\ref{sec:commutativity} proves that dagger limits commute with dagger colimits in a wide range of situations.

To end this introduction, let us discuss earlier attempts at defining dagger limits~\cite{vicary2011completeness}. That work defines a notion of a dagger limit for diagrams $D\colon\cat{J}\to\cat{C}$ where \cat{J} has finitely many objects and \cat{C} is a dagger category enriched in commutative monoids. We dispense with both requirements. Proposition~\ref{prop:limswithsums} shows that when these requirements are satisfied the two notions agree for a wide class of diagrams. However, they do not always agree, as discussed in Example~\ref{ex:pullback}.

\section{Dagger categories}\label{sec:dagcats}

Before going into limits in dagger categories, this section sets the scene by discussing dagger categories themselves. Dagger categories can behave rather differently than ordinary categories, see \eg~\cite[9.7]{hott}. We start by establishing terminology and setting conventions.

\begin{definition}
  A \emph{dagger} is a contravariant involutive identity-on-objects functor. Explicitly a dagger is a functor $\dag\colon\cat{C}\op\to\cat{C}$ satisfying $A^\dag=A$ on objects and $f^{\dag\dag}=f$ on morphisms. A \emph{dagger category} is a category equipped with a dagger. 
\end{definition}

\begin{example}
  Examples of dagger categories abound. We first give some concrete examples, and will later add more abstract constructions.
  \begin{itemize}
    \item 
    Any monoid $M$ equipped with an involutive homomorphism $f\colon M\op\to M$ may be regarded as a one-object dagger category with $x^\dag=f(x)$. For example: the complex numbers with conjugation, either under addition or multiplication. Or: the algebra of (complex) $n$-by-$n$ matrices with conjugate transpose.

    \item 
    The category \cat{Hilb} of (complex) Hilbert spaces and bounded linear maps is a dagger category, taking the dagger of $f \colon A \to B$ to be its adjoint, \ie unique morphism satisfying $\langle f(a) \mid b \rangle = \langle a \mid f^\dag(b) \rangle$ for all $a \in A$ and $b \in B$. The full subcategory of finite-dimensional Hilbert spaces $\cat{FHilb}$ is a dagger category too. 

    \item 
    The category \cat{DStoch} has finite sets as objects. A morphism $A \to B$ is a matrix $f \colon A \times B \to [0,1]$ that is doubly stochastic, \ie satisfies $\sum_{a \in A} f(a,b)=1$ for all $b \in B$ and $\sum_{b \in B} f(a,b)=1$ for all $a \in A$. This becomes a dagger category with $f^\dag(b,a) = f(a,b)$.

    \item 
    In the category $\cat{Rel}$ with sets as objects and relations $R \subseteq A \times B$ as morphisms $A \to B$, composition is given by $S \circ R = \{ (a,c) \mid \exists b \in B \colon (a,b) \in R, (b,c) \in S \}$. This becomes a dagger category with $R^\dag=\{(b,a)\mid (a,b)\in R\}$. 
    The full subcategory $\cat{FinRel}$ of finite sets is also a dagger category.

    \item 
    If $\cat{C}$ is a category with pullbacks, the category $\cat{Span}(\cat{C})$ of spans is defined as follows. Objects are the same as those of $\cat{C}$. A morphism $A \to B$ in $\cat{Span}(\cat{C})$ is a span $A \leftarrow S \to B$ of morphisms in $\cat{C}$, where two spans $A \leftarrow S \to B$ and $A \leftarrow S' \to B$ are identified when there is an isomorphism $S \simeq S'$ making both triangles commute. Composition is given by pullback. The category $\cat{Span}(\cat{C})$ has a dagger, given by $(A \leftarrow S \to B)^\dag = (B \leftarrow S \to A)$.

    \item Any groupoid has a canonical dagger with $f^\dag = f^{-1}$. Of course, a given groupoid might have other daggers as well. For example, the core of \cat{Hilb} inherits a dagger from \cat{Hilb} that is different from the canonical one it has as a groupoid.
  \end{itemize}
\end{example}

In a dagger category, one can define dagger versions of various ordinary notions in category theory. A lot of the terminology for the dagger counterparts of ordinary notions comes from the dagger category \cat{Hilb}.

\begin{definition} 
  A morphism $f \colon A \to B$ in a dagger category is:
  \begin{itemize}
  	\item \emph{unitary} (or a \emph{dagger isomorphism}) if $f^\dag f = \id[A]$ and $ff^\dag=\id[B]$;

  	\item an \emph{isometry} (or a \emph{dagger monomorphism}) if $f^\dag f = \id[A]$;

  	\item a \emph{coisometry} (or a \emph{dagger epimorphism}) if $ff^\dag = \id[B]$;

  	\item a \emph{partial isometry} if $f=ff^\dag f$;
  \end{itemize}
  Moreover, an endomorphism $A\to A$ is
  \begin{itemize}
  	\item \emph{self-adjoint} if $f=f^\dag$;

  	\item a \emph{projection} (or a \emph{dagger idempotent}) if  and $f=f^\dag=ff$;

  	\item \emph{positive} if $f=g^\dag g$ for some $g \colon A \to B$. 
  \end{itemize}
  A \emph{dagger subobject} is a subobject that can be represented by a dagger monomorphism.
  A dagger category is \emph{unitary} when objects that are isomorphic are also unitarily isomorphic.
\end{definition}

A morphism $f$ is dagger monic iff $f^\dag$ is dagger epic. The other concepts introduced above are self-adjoint, \ie a morphism $f$ is unitary/a partial isometry/self-adjoint/a projection/positive iff $f^\dag$ is.
Notice that dagger monomorphisms are split monomorphisms, and dagger epimorphisms are split epimorphisms.
A morphism is dagger monic (epic) if and only if it is both a partial isometry and monic (epic).
In diagrams, we will depict 
partial isometries as $\begin{tikzpicture} \draw[dagger,->] (0,0) to (1,0); \end{tikzpicture}$, 
dagger monomorphisms as $\begin{tikzpicture} \draw[dagger,>->] (0,0) to (1,0); \end{tikzpicture}$, 
dagger epimorphisms as $\begin{tikzpicture} \draw[dagger,->>] (0,0) to (1,0); \end{tikzpicture}$, 
and dagger isomorphisms as $\begin{tikzpicture} \draw[dagger,>->>] (0,0) to (1,0); \end{tikzpicture}$, 
If $f$ is a partial isometry, then both $f^\dag f$ and $ff^\dag$ are projections.

The definition of a dagger subobject might seem odd: two monomorphisms $m:M\hookrightarrow A$ and $n\colon N\hookrightarrow A$ are considered to represent the same subobject when there is an isomorphism $f\colon M\to N$ such that $m=nf$, whereas one might expect that dagger monics $m$ and $n$ represent the same \emph{dagger} subobject if $f$ can be chosen to be unitary. Hence there seems to be a possibility that two dagger monics representing the same subobject might nevertheless represent different dagger subobjects. However, this can not happen: if $m,n$ are dagger monic and $f$ is an isomorphism satisfying $m=nf$, then $f$ is in fact unitary. This is because $m=nf$ implies  $n^\dag m=n^\dag n f=f$ and $mf^{-1}=n$ implies $f^{-1}=m^\dag n$ so that $f^{-1}=f^\dag$.

\begin{example}
  An \emph{inverse category} is a dagger category in which every morphism is a partial isometry, and positive morphisms commute~\cite{kastl:inversecategories}: $ff^\dag f=f$ and $f^\dag f g^\dag g = g^\dag g f^\dag f$ for all morphisms $f \colon A \to B$ and $g \colon A \to C$. A prototypical inverse category is \cat{PInj}, the category of sets and partial injections, which is a subcategory of $\cat{Rel}$. 	
\end{example}

\begin{remark}\label{rem:pisetc} 
  Partial isometries need not be closed under composition~\cite[5.4]{heunen:ltwo}. The same holds for self-adjoint morphisms and positive morphisms. 
  However, given partial isometries $p\colon A\to B$ and $q\colon B\to C$, their composition $qp$ is a partial isometry whenever the projections $q^\dag q$ and $pp^\dag$ commute. Similarly, a morphism $f$ factors through a partial isometry $p$ with the same codomain if and only if $pp^\dag f=f$.
\end{remark}

\begin{definition}
  A \emph{dagger functor} is a functor $F\colon\cat{C}\to \cat{D}$ between dagger categories satisfying $F(f^\dag)=F(f)^\dag$. Denote the category of small dagger categories and dagger functors by \cat{DagCat}. 
\end{definition}

There is no need to go further and define `dagger natural transformations': if $\sigma\colon F\Rightarrow G$ is a natural transformation between dagger functors, then taking daggers componentwise defines a natural transformation $\sigma^\dag\colon G\Rightarrow F$. This no longer holds for natural transformations between arbitrary functors into a dagger category, though, see Definition~\ref{def:adjointability}.

\begin{example}\label{ex:functorcat}
  If $\cat{C}$ and $\cat{D}$ are dagger categories, then the category $\cat{DagCat}(\cat{C},\cat{D})$ of dagger functors $\cat{C} \to \cat{D}$ and natural transformation is again a dagger category.
  In particular, taking $\cat{C}$ to be a group $G$ and $\cat{D}=\cat{Hilb}$, this shows that the category of unitary representations of $G$ and intertwiners is a dagger category.
\end{example}

\begin{example}
  Dagger categories also form the objects of a dagger category as follows~\cite[3.1.8]{heunen:thesis}: morphisms $\cat{C} \to \cat{D}$ are functors $F \colon \cat{C}\op \to \cat{D}$ that have a left adjoint, where two such functors are identified when they are naturally isomorphic. The identity on $\cat{C}$ is its dagger $\dag \colon \cat{C}\op \to \cat{C}$; its left adjoint is $\dag\op \colon \cat{C}\op \to \cat{C}$. The composition of $F \colon \cat{C}\op \to \cat{D}$ and $G \colon \cat{D}\op \to \cat{E}$ is $G \circ \dag \circ F \colon \cat{C}\op \to \cat{E}$; its left adjoint is $F' \circ \dag \circ G'$, where $F' \dashv F$ and $G' \dashv G$.
  The dagger of $F \colon \cat{C}\op \to \cat{D}$ is given by the right adjoint of $F\op \colon \cat{C} \to \cat{D}\op$.
\end{example}

Example~\ref{ex:functorcat} in fact makes \cat{DagCat} into a \emph{dagger 2-category}. Formally, this means a category enriched in \cat{DagCat}, Concretely, a dagger 2-category is a 2-category where the 2-cells have a dagger that cooperates with vertical and horizontal composition, in the sense that the equations
    \begin{align*}
      (\sigma^{\dag})^\dag&=\sigma \\
      (\tau\circ\sigma)^\dagger&=\sigma^\dagger\circ\tau^\dagger \\
      (\sigma*\tau)^\dag&=(\sigma^\dag) * (\tau^\dag)
    \end{align*}
hold, where $\circ$ denotes vertical composition and $*$ denotes horizontal composition.
Dagger 2-functors are 2-functors that preserve this dagger. Strictly speaking, this defines \emph{strict} dagger 2-categories. Defining dagger bicategories is straightforward but not necessary for our purposes.

The forgetful 2-functor $\cat{DagCat}\to\cat{Cat}$ has no 2-adjoints, but the forgetful 1-functor has both 1-adjoints.
As final examples of dagger categories, we recall the definitions of free and cofree dagger categories~\cite[3.1.17,3.1.19]{heunen:thesis}, which will be used in Examples~\ref{ex:daggershaped} and~\ref{ex:cofree} below.

\begin{proposition}\label{prop:free}
  The forgetful functor $\cat{DagCat} \to \cat{Cat}$ has a left adjoint \cat{ZigZag(-)}. The objects of \cat{Zigzag}(\cat{C}) are the same as in \cat{C}, and a morphism $A \to B$ is an alternating sequence of morphisms $A \rightarrow C_1 \leftarrow \cdots \rightarrow C_n \leftarrow B$ from $\cat{C}$, subject to the identifications $\big(A \sxto{f} C \sxfrom{\id} C \sxto{g} D \sxfrom{h} B \big) = \big( A \sxto{g \circ f} D \sxfrom{h} B \big)$ and $\big(A \sxto{f} C \sxfrom{g} D \sxto{\id} D \sxfrom{h} B\big) = \big( A \sxto{f} C \sxfrom{g \circ h} B\big )$; composition is given by juxtaposition. The category $\cat{Zigzag}(\cat{C})$ has a dagger $(f_1,\ldots,f_n)^\dag = (f_n,\ldots,f_1)$.
\end{proposition}

\begin{proposition}\label{prop:cofree}
  The forgetful functor $\cat{DagCat} \to \cat{Cat}$ has a right adjoint $(-)_\leftrightarrows$,
  which sends a category $\cat{C}$ to the full subcategory of $\cat{C}\op \times \cat{C}$ with objects of the form $(A,A)$, 
  and sends a functor $F$ to the restriction of $F\op \times F$. The dagger on $\cat{C_\leftrightarrows}$ is given by $(f,g)^\dag=(g,f)$.
\end{proposition}

\begin{definition}\label{def:adjointability}
  Consider a natural transformation $\sigma\colon F\Rightarrow G$ where $F,G$ are arbitrary functors with codomain a dagger category \cat{C}. Taking daggers componentwise generally does not yield a natural transformation $G\Rightarrow F$. When it does, we call $\sigma$ \emph{adjointable}. For small \cat{J}, we denote the dagger category of functors $\cat{J}\to\cat{D}$ and adjointable natural transformations by $[\cat{J},\cat{C}]_\dag$.
\end{definition}

A natural transformation $\sigma\colon F\Rightarrow G$ is adjointable if and only if $\sigma$ defines a natural transformation $\dag\circ F\Rightarrow \dag\circ G$

\begin{lemma}\label{lem:daggersubfunctorsaredagger} 
  If $G$ is a dagger functor and $\sigma\colon F\Rightarrow G$ is a natural transformation that is pointwise dagger monic, then $F$ is a dagger functor as well.
\end{lemma}
\begin{proof}
  As $\sigma$ is pointwise dagger monic and natural, the following diagram commutes for any $f$:
  \[\begin{tikzpicture}
   	\matrix (m) [matrix of math nodes,row sep=2em,column sep=4em,minimum width=2em]
   	{
    	FA & FB & \\
    	GA & GB \\};
   	\path[->]
   	(m-1-1) edge[dagger,>->] node [left] {$\sigma$} (m-2-1)
      	     edge node [above] {$F(f)$} (m-1-2)
   	(m-2-1) edge node [below] {$G(f)$} (m-2-2)
   	(m-2-2) edge[dagger,->>] node [right] {$\sigma^\dag$} (m-1-2);
  \end{tikzpicture}\]
  Taking the dagger of this diagram, and replacing $f$ with $f^\dag$, respectively, results in two commuting diagrams:
  \[\begin{tikzpicture}
  	\matrix (m) [matrix of math nodes,row sep=2em,column sep=4em,minimum width=2em]
   	{
    	FB & FA & \\
    	GB & GA \\};
   	\path[->]
   	(m-1-1) edge[dagger,>->] node [left] {$\sigma$} (m-2-1)
           edge node [above] {$F(f)^\dag$} (m-1-2)
   	(m-2-1) edge node [below] {$G(f)^\dag$} (m-2-2)
   	(m-2-2) edge[dagger,->>] node [right] {$\sigma^\dag$} (m-1-2);
  	\end{tikzpicture}
  	\quad
   	\begin{tikzpicture}
  	\matrix (m) [matrix of math nodes,row sep=2em,column sep=4em,minimum width=2em]
  	{
   	FB & FA & \\
   	GB & GA \\};
  	\path[->]
  	(m-1-1) edge[dagger,>->] node [left] {$\sigma$} (m-2-1)
     	     edge node [above] {$F(f^\dag)$} (m-1-2)
  	(m-2-1) edge node [below] {$G(f^\dag)$} (m-2-2)
  	(m-2-2) edge[dagger,->>] node [right] {$\sigma^\dag$} (m-1-2);
  \end{tikzpicture}\]
  Because $G$ is a dagger functor, the bottom arrows in both diagrams are equal. Therefore top arrows are equal too, making $F$ a dagger functor.
\end{proof}

\section{Dagger limits}\label{sec:daglims}

Before we define dagger limits formally, let us look at few examples to motivate the definition. If \cat{1} is the terminal category, then a functor $D\colon \cat{1}\to \cat{Hilb}$ is given by a choice of an object $H$. Now, a limit of $D$ is is given by a an object $H'$ and an isomorphism $f\colon H'\to H$. However, for a limit of $D$ to qualify as the dagger limit of $D$, one would expect the isomorphism $f$ to be unitary. However, for some diagrams it seems as if one needs to make choices: a limit of $2\times -\colon \mathbb{C}\leftrightarrows\mathbb{C}\colon-/2$ is given by an isomorphism to one (and hence to both) copies of $\mathbb{C}$. For a dagger limit of the same diagram, we can require one of the isomorphisms to be unitary but not both -- which copy of $\mathbb{C}$ one should prefer? 

We will discuss more instructive examples below, but already it seems like dagger limits have something to do with ``normalization'', and one can conceivably choose to normalize at different locations. To keep track of these choices, we build them in the definition below, which is the basic object of study in this chapter. The rest of this section illustrates it.

\begin{definition}\label{def:daglim}
  Let \cat{C} be a dagger category and \cat{J} a category. A class $\Omega$ of objects of \cat{J} is \emph{weakly initial} if for every object $B$ of \cat{J} there is a morphism $f\colon A\to B$ with $A\in\Omega$, \ie if every object of \cat{J} can be reached from $\Omega$. 
  Let $D\colon\cat{J}\to\cat{C}$ be a diagram and let $\Omega\subseteq \cat{J}$ be weakly initial. \emph{A dagger limit of $(D,\Omega)$} is a limit $L$ of $D$ whose cone $l_A \colon L \to D(A)$ satisfies the following two properties:
  \begin{description}
    \item[normalization] $l_A$ is a partial isometry for every $A \in \Omega$;
    \item[independence] the projections on $L$ induced by these partial isometries commute, \ie $l_A^\dag l_A l_B^\dag l_B=l_B^\dag l_Bl_A^\dag l_A$ for all $A,B\in \Omega$.
  \end{description}
  A dagger limit of $D$ is a dagger limit of $(D,\Omega)$, for some weakly initial $\Omega$.
  If $L$ is a dagger limit of $(D,\Omega)$, we will also write $L=\dlim^\Omega D$. For a fixed $\cat{J}$ and $\Omega$, if $(D,\Omega)$ has a dagger limit in \cat{C} for every $D$, we will say that \cat{C} has all $(\cat{J},\Omega)$-shaped limits.
\end{definition}

Note that if $L$ is a dagger limit of $(D,\Omega)$ and $\Psi\subset\Omega$ is weakly initial, $L$ is also a dagger limit of $(D,\Psi)$. Moreover, if $L$ is a dagger limit of $(D,\Psi)$ and $(D,\Omega)$ also has a dagger limit, Theorem~\ref{thm:daglimsunique} will imply that $L$ is also a dagger limit of $(D,\Omega)$.

\begin{example}\label{ex:concretedaggerlimits}
  Definition~\ref{def:daglim} subsumes various concrete dagger limits from the literature:
  \begin{itemize}
    \item A terminal object is a limit of the unique functor $\emptyset\to\cat{C}$. As the empty category has no objects, being a dagger limit of $\emptyset\to\cat{C}$ says nothing more than being terminal. In a dagger category, any terminal object is automatically a zero object.

    \item A \emph{dagger product} of objects $A$ and $B$ in a dagger category with a zero object is traditionally defined~\cite{selinger:completelypositive} to be a product $A \times B$ with projections $p_A \colon A \times B \to A$ and $p_B \colon A \times B \to B$ satisfying $p_A p_A^\dag = \id[A]$, $p_B p_B^\dag = \id[B]$, $p_A p_B^\dag = 0$, and $p_B p_A^\dag = 0$.
    This is precisely a dagger limit of $(D,\Omega)$, where $\cat{J}$ is the discrete category on two objects that $D$ sends to $A$ and $B$, and $\Omega$ necessarily consists of both objects of $\cat{J}$; see also Example~\ref{ex:biproducts} below.

    \item A \emph{dagger equalizer} of morphisms $f,g \colon A \to B$ in a dagger category is traditionally defined~\cite{vicary2011completeness} to be an equalizer $e \colon E \to A$ that is dagger monic.
    This is precisely a dagger limit, where $\cat{J} = \bullet \rightrightarrows \bullet$ which $D$ sends to $f$ and $g$, and $\Omega$ consists of only the first object, which gets sent to $A$. 

    This example justifies why Definition~\ref{def:daglim} cannot require $\Omega$ to be all of $\cat{J}$ in general, otherwise there would be many pairs $f,g$ that have a dagger equalizer in the traditional sense but not in the sense of Definition~\ref{def:daglim}.

    \item A \emph{dagger kernel} of a morphism $f \colon A \to B$ in a dagger category with a zero object is traditionally defined~\cite{heunenjacobs:daggerkernels} to be a kernel $k \colon K \to A$ that is dagger monic. As a special case of a dagger equalizer it is a dagger limit.

    \item A \emph{dagger intersection} of dagger monomorphisms $f_i \colon A_i \to B$ in a dagger category is traditionally defined~\cite{vicary2011completeness} to be a (wide) pullback $P$ such that each leg $p_i \colon P \to A_i$ of the cone is dagger monic.
    This is precisely a dagger limit, where $\Omega$ consists of all the objects of $\cat{J}$ getting mapped to $A_i$. Since pullback of monics are monic, each $p_i$ is not only a partial isometry but also a monomorphism, and hence a dagger monomorphism.

    \item If $p\colon A\to A$ is a projection, a \emph{dagger splitting of $p$} is a dagger monic $i\colon I\to A$ such that $p=ii^\dag$~\cite{selinger2008idempotents}. A dagger splitting of $p$ can be seen as the dagger limit of the diagram generated by $p$. 
    More precisely, we can take $\Omega=\{A\}$, by definition $i$ is a partial isometry, and if $l \colon L \to A$ is another limit, then $m=i^\dag l$ is the unique map satisfying $l=im$.
    Conversely, suppose that $l \colon L \to A$ is a dagger limit. Then $l$ is a partial isometry, and so the cone $l$ factors through itself via both $l^\dag l$ and $\id[L]$; but since mediating maps are unique these must be equal, and so $l$ is dagger monic. Similarly, because $p$ is idempotent, $p$ gives a cone, which factors through $l$. This implies $l l^\dag p=p$. Taking daggers we see that $p=pll^\dag=ll^\dag$ since $pl=l$. 
    We say that \cat{C} has dagger splittings of projections if every projection has a dagger splitting. 
  \end{itemize}
\end{example}

\begin{example}\label{ex:daggershaped}
  Let \cat{J} be a dagger category, and $D\colon \cat{J}\to\cat{C}$ a dagger functor. Any leg $l_A \colon L \to D(A)$ of a cone is a partial isometry if and only if any other leg $l_B \colon L \to D(B)$ in the same connected component is. To see this, fix a morphism $f \colon A \to B$ in $\cat{J}$, and assume that $l_{A}$ is a partial isometry. 
  \begin{align*} 
    l_Bl_B^\dag l_B
    &=D(f) l_A l_A^\dag  D(f)^\dag D(f)l_A && \text{ as }L\text{ is a cone}\\
    &=D(f) l_A l_A^\dag  D(f^\dag f)l_A && \text{ as }D\text{ is a dagger functor}\\
    &=D(f) l_A l_A^\dag l_A && \text{ as }L\text{ is a cone}\\
    &=D(f) l_A && \text{ as }l_A\text{ is a partial isometry} \\
    &=l_B  && \text{ as }L\text{ is a cone}
  \end{align*}
  Similarly $l_A^\dag l_A=l_B^\dag l_B$ for any objects $A$ and $B$ in the same connected component of \cat{J}. These two facts imply that whenever $D$ is a dagger functor, the choice of the parameter $\Omega$ doesn't matter when speaking about dagger limits, as the resulting equations are equivalent. Hence whenever $D$ is a dagger functor we omit $\Omega$, and call a dagger limit of $D$ a \emph{dagger-shaped limit}. In particular, whenever every dagger functor $\cat{J}\to\cat{C}$ has a dagger limit, we will say that \cat{C} has \cat{J}-shaped limits. If \cat{C} has \cat{J}-shaped limits for every small dagger category \cat{J}, we will say that \cat{C} has dagger-shaped limits. If we wish to say something about dagger limits of  functors $\cat{J}\to\cat{C}$ that don't necessarily preserve the dagger, we will make it clear by not omitting $\Omega$.
  \begin{itemize}
    \item Any discrete category has a unique dagger, which is always preserved by maps into dagger categories. Thus dagger products can be seen as dagger-shaped limits. 

    \item Any dagger splitting of a projection is a dagger-shaped limit, as in Example~\ref{ex:concretedaggerlimits}.

    \item We say that a dagger category has \emph{dagger split infima of projections} if, whenever $\mathcal{P}$ is a family of projections on a single object $A$, it has an infimum admitting a dagger splitting, \ie a dagger subobject $K\rightarrowtail A$ such that the induced projection is the infimum of $\mathcal{P}$. Limits of projections can be defined as dagger-shaped limits: consider the monoid  freely generated by a set of idempotents; the dagger on the monoid fixes those idempotents, and reverses words in them. However, we prefer to think of them instead in terms of the partial order on projections. It is not hard to show that the dagger intersection of a family of dagger monics $m_i\colon A_i\to A$ coincides with the dagger limit of the projections $m_im_i^\dag\colon A\to A_i\to A$.

    \item We say that a dagger category has \emph{dagger stabilizers} when it has all $\cat{ZigZag(E)}$-shaped limits, where \cat{E} is the equalizer shape and $\cat{ZigZag}(E)$ is the free dagger category on \cat{E} from Proposition~\ref{prop:free}. Concretely, a dagger functor with domain $\cat{ZigZag(E)}$ is uniquely determined by where it sends $E$, \ie by a choice of a parallel morphisms $f,g\colon A\rightrightarrows B$ in the target category \cat{C}. A cone for such a functor consists of an object $X$ with maps $p_A\colon X\to A$ and $p_B$ satisfying $fp_A=gp_A=p_B$ \emph{and} $f^\dag p_B=g^\dag p_B=p_A$. A dagger stabilizer of $f$ and $g$ is then a terminal such cone that also satisfies normalization and independence. Hence the dagger stabilizer of $f$ and $g$ is \emph{not} in general a dagger equalizer. For example, the (dagger) kernel of a linear map $f\colon A\to B$ in \cat{FHilb} can be computed as the equalizer of $f$ and $0$, whereas the dagger stabilizer of $f$ and $0$ is always $0$.
  \end{itemize}
\end{example}

Recall that a dagger category is \emph{connected} if every hom-set is inhabited. 
\begin{proposition}\label{prop:variouscatshavedagshapedlims}
  The dagger categories $\cat{Rel}$, $\cat{FinRel}$, $\cat{PInj}$, and $\cat{Span(FinSet)}$ have \cat{J}-shaped limits for any small connected dagger category \cat{J}.
\end{proposition}
\begin{proof} 
  Let $D\colon \cat{J}\to \cat{Rel}$ be a dagger functor; we will construct a dagger limit. Write $G$ for the (undirected multi-)graph with vertices $V=\coprod_{A\in\cat{J}} D(A)$ and edges $E=\coprod_{f\in\cat{J}}D(f)$. Call a vertex $a\in D(A)\subseteq V$ a \emph{$D$-endpoint} if $D(f)a=\emptyset$ for some $f\colon A\to B$ in \cat{J}. Set
  \[
    L=\{X\in \mathcal{P}(V)\mid X\text{ is a path component of }\mathcal{G}\text{ with no }D\text{-endpoints}\}\text,
  \]
  and define $l_A\colon L\to D(A)$ by $l_A(X)=X\cap D(A)$. We will prove that this is a dagger limit of $D$, starting with normalization and independence. First we show that $l_A(X)\neq \emptyset$ for any $A\in\cat{J}$ and $X\in L$. Pick an element $b\in X$, say $b\in D(B)$, and choose some $f\colon B\to A$. Since $x$ is not a $D$-endpoint, $D(f)x$ is nonempty and contained in $X\cap A=l_A(X)$. Because path components of a graph are disjoint, $l_A^\dag l_A (X)=X$ for all $X$. Hence $l_A^\dag l_A=\id[L]$ for all $A$, establishing normalization and independence.

  Next we verify that $l_A$ forms a cone. Path components are closed under taking neighbours, so $D(f)l_A(X)\subseteq D_B(X)$. To see the other inclusion, let $b\in D(B)(X)=X\cap B$. Again $b$ isn't a $D$-endpoint, so $D(f)^\dag (b)=D(f^\dag) (b)\subseteq L_A(X)$ is nonempty. Any element of $D(f)^\dag (b)$ is related to $b$ by $D(f)$. Hence $b\in D(f)l_A(X)$, so $D(f)l_A=l_B$ as desired.

  Finally, we verify that the cone $l_A$ is limiting. Let $R_A\colon Y\to D(A)$ be any cone. If $x\in D(A)$ is a $D$-endpoint, say $D(f)x=\emptyset$, then 
  \[
    x\notin D(f^\dag)D(f)D(A)=D(f^\dag f)D(A)\supseteq D(f^\dag f)R_A(y)=R_A(y)\text.
  \]
  Hence no $R_A(y)$ contains $D$-endpoints. Moreover, if $D(f)R_A(y)=R_B(y)$ for all $f$, then the set $\coprod_{A\in\cat{J}}(R_A(y))\subseteq V$ is closed under taking neighbours in $G$. As it contains no $D$-endpoints, it is a union of a set of connected components of $\mathcal{G}$ without $D$-endpoints. Mapping $y$ to this set of connected components, \ie\ to a subset of $L$, defines the unique relation $R\colon Y\to L$ satisfying $R_A=l_A\circ R$ for all $A\in\cat{J}$.

  The same construction works for \cat{FinRel}:  one merely needs to check that $L$ is finite whenever each $D(A)$ is. In fact $L$ is finite if at least one $D(A)$ is, since $l_A^\dag l_A=\id[L]$ implies that the function $L\to \mathcal{P}(D(A))$ corresponding to $l_A$ is injective.

  Now consider a dagger functor $D\colon\cat{J}\to\cat{PInj}$ and set
  \begin{align*}
    L&=\{(x_A)_{A\in\cat{J}}\in \prod_{A\in J}D(A)\mid D(f)x_A=x_B\text{ for every }f\colon A\to B\}\text, \\
    l_A((x_A)_{A\in\cat{J}})&=x_A\text.
  \end{align*}
  It is easy to verify that this forms a dagger limit. 

  Dagger limits of a dagger functor $D\colon\cat{J}\to\cat{Span(FinSet)}$ resemble the case of \cat{PInj} more than \cat{(Fin)Rel}.  Think of (the isomorphism class of) a span $A\leftarrow \bullet\to B$ of finite sets as a matrix $R\colon A\times B\to\mathbb{N}$ with natural number entries, so that a morphism $f\colon A\to B$ in \cat{J} maps to $D(f)\colon D(A)\times D(B)\to\mathbb{N}$. Set 
    \begin{align*}
    L&=\{(x_A)_{A\in\cat{J}}\in \prod_{A\in J}D(A)\mid D(f)(x_A,z)=\delta_{z,x_B}
      \text{ for every }f\colon A\to B\}\text, \\
    l_A((x_A)_{A\in\cat{J}},z)&=\delta_{z,x_A}\text.
  \end{align*}
  It is easy to see $l_A$ forms a cone. To see that it is limiting, let $R_A\colon Y\to D(A)$ be any cone, and pick $y\in Y$.  Consider $x_A\in R_A$ such that $R_A(y,x_A)=n\neq 0$. We will show that if $B\in \cat{J}$ and $f\colon A\to B$, then $D(f)(x_A,z)=\delta_{z,x_B}$ for some unique $x_B$, so that $x_A$ extends to a unique family $(x_A)_{A\in\cat{J}}\in L$. Consider an arbitrary $f\colon A\to B$. Now $D(f^\dag f) R_A=R_A$, so there has to be some $x_B\in B$ with $D(f)(x_A,x_B)>0.$ If there were several such $x_B$ or if $D(f)(x_A,x_B)>1$, then $D(f^\dag f)R_A(y,x_A)>n$, which is a contradiction. Hence $x_A$ extends uniquely to a family $(x_A)_{A\in\cat{J}}\in L$. Moreover, if $R_B=D(f)R_A$, then $R_B(y,x_B)=n$ as well. Hence we can define $R\colon Y\to L$ by setting $R(y,(x_A)_{A\in\cat{J}})=R_A(y,x_A)$. Now $R$ satisfies $l_BR=R_B$ for each $B$ and it is clearly unique as such.
\end{proof}

Note that this theorem fails for \cat{Span(Set)}, since idempotents do not always split. For instance, the idempotent $1\leftarrow \mathbb{N}\to 1$ does not admit a splitting. We leave open the question of characterizing exactly which categories of spans or relations admit connected dagger-shaped limits.

The following example illustrates the name `independence axiom' in Definition~\ref{def:daglim}.

\begin{example} 
  When working in \cat{FHilb}, consider $\mathbb{C}^2$ as the sum of two non-orthogonal lines, \eg the ones spanned by $\lvert 0 \rangle$ and $\lvert + \rangle$. Projections to these two lines will give rise to two maps $p_1,p_2\colon \mathbb{C}^2\to \mathbb{C}$ making $(\mathbb{C}^2,p_1,p_2)$ into a categorical product. Moreover, $p_1$ and $p_2$ are partial isometries so that the normalization axiom is satisfied. However, the independence axiom fails, and indeed, $(\mathbb{C}^2,p_1,p_2)$ fails to be a dagger product. In other words, the limit structure $(\mathbb{C}^2,p_1,p_2)$ and the colimit structure $(\mathbb{C}^2,p_1^\dag,p_2^\dag)$ are not compatible.
\end{example}

\begin{example}\label{ex:cofree}
  Sometimes in ordinary category theory an object is both a limit and a colimit ``in a compatible way'' to a pair of related diagrams in \cat{C}. Usually this is formulated in terms of a canonical morphism from the colimit to the limit being an isomorphism, but in some cases one can instead formulate them as dagger limits in $\cat{C_\leftrightarrows}$, the cofree dagger category from proposition \ref{prop:cofree}. Moreover, if \cat{C} has zero morphisms\footnote{This is so that one can give every cone a trivial cocone structure and vice versa.}, then dagger limits in $\cat{C_\leftrightarrows}$  give rise to such ``ambilimits '' in \cat{C}. 
  \begin{itemize}
    \item If \cat{C} has zero morphisms, then a biproduct in \cat{C} is the same thing as a dagger product in \cat{C_\leftrightarrows}. 

    \item Idempotents in \cat{C} split if and only if (dagger) projections in $\cat{C_\leftrightarrows}$ have dagger splittings.

    \item A more interesting example comes from domain theory, where there is an important limit-colimit coincidence. Regard the partially ordered set $(\mathbb{N},\leq)$ of natural numbers as a category $\omega$, and let $D\colon \omega\to \cat{DCPO}$ be a chain of embeddings, \ie each map in $\omega$ is mapped to an embedding by $D$. Then the embeddings define unique projections, resulting in a chain of projections $D^*\colon\omega\op\to \cat{DCPO}$. A fundamental fact in domain theory~\cite[3.3.2]{abramsky1994domain} is that the colimit of $D$ coincides with the limit of $D^*$. One can go through the construction and show that this ``ambilimit'' can equivalently be described as the dagger limit of $(D,D^*,\omega)$, where $(D^*,D)\colon\omega\op\to\cat{DCPO}_\leftrightarrows$.
  \end{itemize}
  This viewpoint is developed in more detail in~\cite[Chapter 5]{karvonen:thesis}.
\end{example}

\begin{example}
  In an inverse category the normalization and independence axioms of Definition~\ref{def:daglim} are automatically satisfied, and hence dagger limits are simply limits.
\end{example}

\begin{example} 
  Fix $\cat{J}$ and a weakly initial $\Omega$. If $\cat{C}$ has $(\cat{J},\Omega)$-shaped limits, so does $[\cat{D},\cat{C}]$.
  Let $D \colon \cat{J} \to [\cat{D},\cat{C}]$ be a diagram.
  For each $X \in \cat{D}$, there is a dagger limit $L(X) = \dlim^\Omega D(-)(X)$ with cone $l_A^X \colon L(X) \to D(A)(X)$. 
  For $f \colon X \to Y$ in $\cat{D}$, there is a cone $D(A)(f) \circ l_A^X \colon L(X) \to D(A)(Y)$, and hence a unique map $L(f) \colon L(X) \to L(Y)$ satisfying $l^Y_A \circ L(f) = D(A)(f) \circ l^X_A$. 
  The resulting functor $L \colon \cat{D} \to \cat{C}$ is a limit of $D$, with cone $l_A^X \colon L(X) \to D(A)(X)$.
  This limit is in fact a dagger limit $(D,\Omega)$, because the normalization and independence axioms hold for each component $l_A^X$, and the dagger in $[\cat{D},\cat{C}]$ is computed componentwise.
\end{example}

We end this section by recording how dagger functors interact with dagger limits.
Any dagger functor preserves dagger limits as soon as it preserves limits.
The same holds for reflection and creation of (dagger) limits when the functor is faithful.

\begin{lemma}
  Let $\cat{F} \colon \cat{C} \to \cat{D}$ be a dagger functor. If $F$ preserves limits of type $\cat{J}$, then it preserves $(\cat{J},\Omega)$-shaped dagger limits. If $F$ reflects (creates) limits of type \cat{J}, then it reflects (creates) $(\cat{J},\Omega)$-shaped dagger limits.
\end{lemma}
\begin{proof}
  All dagger functors preserve partial isometries and commutativity of projections, and faithful dagger functors also reflect these. 
\end{proof}

\section{Uniqueness up to unitary isomorphism}\label{sec:unitarity}

\begin{theorem}\label{thm:daglimsunique}
  Let $D\colon\cat{J}\to \cat{C}$ be a diagram and $\Omega\subseteq\cat{J}$ be weakly initial. Let $L$ be a dagger limit of $(D,\Omega)$ and let $M$ be a limit of $D$. The canonical isomorphism of cones $L\to M$ is unitary iff $M$ is a dagger limit of $(D,\Omega)$. In particular, the dagger limit of $(D,\Omega)$ is defined up to unitary isomorphism.
\end{theorem}

\begin{example}\label{ex:biproducts}
  Let \cat{J} be a discrete category of arbitrary cardinality. As \cat{J} has only one weakly initial class (the one consisting of all objects), the dagger limit of any diagram $D\colon \cat{J}\to\cat{C}$ is unique up to unitary isomorphism. These are exactly the dagger products. However, note that Definition~\ref{def:daglim} does not require enrichment in commutative monoids nor the equation
  \begin{equation}\label{eq:biproduct}
    \id[A\oplus B]=l_A^\dag l_A+l_B^\dag l_B
  \end{equation} 
  (and thus works for infinite \cat{J} as well). Moreover, it doesn't require \cat{C} to have a zero object or zero morphisms in order to be defined up to unitary iso. On the other hand, if \cat{C} has zero morphisms, it is not hard to show directly that a dagger product in the sense of Definition~\ref{def:daglim} satisfies the traditional equations involving zero. Moreover, if \cat{C} is enriched in commutative monoids, equation~\eqref{eq:biproduct} also follows. These facts are proven for more general \cat{J} in Propositions~\ref{prop:limswithzeros} and~\ref{prop:limswithsums}.

  In fact, from Definition~\ref{def:daglim} one can glean a definition of (ordinary) biproduct of $A$ and $B$ that is unique up to isomorphism in an ordinary category \cat{C}, but doesn't require the existence of zero morphisms, and so generalizes the usual definition. For example, in \cat{Set} the biproduct $\emptyset\oplus\emptyset$ exists and is the empty set. Slightly more interestingly, if \cat{C} has all binary biproducts (in the traditional sense) and \cat{D} is any non-empty category, then $\cat{C}\sqcup \cat{D}$ has binary biproducts of pairs of objects from \cat{C} (in the generalized sense), but doesn't have zero morphisms. Of course, if a category has all binary biproducts in this generalized sense, one can show that it also has zero morphisms, so this definition is more general only in categories with some but not all biproducts. 
  For more details, see~\cite{karvonen:biproducts}.
\end{example}

\begin{example}
  Let \cat{J} be the indiscrete category on $n$ objects. When considering functors $\cat{J}\to\cat{C}$ that don't preserve the dagger on $\cat{J}$, the parameter $\Omega$ matters. 
  For example, take $n=2$ and consider the diagram $D \colon \cat{J} \to \cat{Hilb}$ defined by $D(1)=D(2)=\mathbb{C}$ where $D(1 \to 2)$ multiplies by $2$ but $D(2 \to 1)$ divides by $2$.
  Now $\Omega$ cannot be all of $\{1,2\}$, because no limiting cone can consist of partial isometries.
  If $\Omega = \{1\}$, there is a dagger limit $L=\mathbb{C}$ with $l_1=1$ and $l_2=2$.
  If $\Omega = \{2\}$, there is a dagger limit $L=\mathbb{C}$ with $l_1=\tfrac{1}{2}$ and $l_2=1$.
  These two dagger limits are clearly not unitarily isomorphic.

  The same can happen with \emph{chains}, where the preorder of integers is regarded as a category $\cat{J}$.
  Consider the diagram $D \colon \cat{J} \to \cat{Hilb}$ defined by $D(n) = \mathbb{C}$, and $D(n \to n+1)$ is multiplication by $-1^{n}$. 
  Now $\Omega$ can either consist of even numbers or the odd numbers. Hence $D$ has two dagger limits and they are not unitarily isomorphic.

  One might hope to get rid of this dependence on $\Omega$ by strengthening the definition to select exactly one of a diagram's several dagger limits. However, this is impossible in general. Write $\cat{J}(X)$ for the indiscrete category on a nonempty set $X\subset \mathbb{R}\setminus\{0\}$ of objects. Define $D(X)\colon\cat{J}(X)\to\cat{FHilb}$ by mapping the unique arrow $x\to y$ to the morphism $\mathbb{C}\to\mathbb{C}$ that multiplies by $\tfrac{x}{y}$. A choice of a dagger limit for each $D(X)$ amounts to a choice function on $\mathbb{R}\setminus\{0\}$. Thus there is no way to strengthen Definition~\ref{def:daglim} to make dagger limits unique in a way that doesn't depend on a choice of a weakly initial class.
\end{example}

\begin{example}
  In the domain theory part of Example~\ref{ex:cofree}, in fact $\Omega=\cat{J}$. Hence the bilimit is unique up to unique unitary isomorphism in \cat{DCPO_\leftrightarrows}. In  \cat{DCPO}, this means that any isomorphism of the limit half of such bilimits, is also an isomorphism of the colimit half.
\end{example}

\begin{example} 
  Consider dagger equalizers of $f,g\colon A\to B$ in the sense of Definition~\ref{def:daglim}. Any weakly initial class must contain $A$, and thus the dagger equalizer is unique up to unitary iso if it exists. Moreover, it is readily seen to coincide with the traditional definition. For if $e\colon E\to A$ is the dagger equalizer in the above sense, then $e$ is monic and a partial isometry, and thus dagger monic, so that $e$ is a dagger equalizer in the traditional sense.
\end{example}

\begin{theorem}\label{thm:equivalenttounitary} 
  Let $l_A \colon L \to D(A)$ and $m_A \colon M \to D(A)$ be limits of the same diagram $D \colon \cat{J} \to \cat{C}$ where \cat{C} is a dagger category, and let $f\colon L\to M$ be the unique isomorphism of limits. Then the following conditions are equivalent:
  \begin{enumerate}[(i)]
    \item $f$ is unitary;
    \item $f$ is also a morphism of colimits;
    \item the following diagram commutes for any $A$ and $B$ in $\cat{J}$:
    \[\begin{tikzpicture}
     \matrix (m) [matrix of math nodes,row sep=2em,column sep=4em,minimum width=2em]
     {
      D(A) & L \\
      M & D(B) \\};
     \path[->]
     (m-1-1) edge node [left] {$m_A^\dag$} (m-2-1)
            edge node [above] {$l_A^\dag$} (m-1-2)
     (m-1-2) edge node [right] {$l_B$} (m-2-2)
     (m-2-1) edge node [below] {$m_B$} (m-2-2);
    \end{tikzpicture}\]
  \end{enumerate}
\end{theorem}
\begin{proof}
    $(i)\Rightarrow (ii)$ By definition $f^{-1}$ is the unique map $M\to L$ that is compatible with the limit structure, whereas $f^{\dag}$ is the unique map $M\to L$ that is compatible with the \emph{colimit} structure. As $f$ is unitary, these coincide, whence $f^{-1}$ is simultaneously a map of limits and colimits. Therefore so too is $f$.

    $(ii)\Rightarrow (iii)$  By (ii), both of the triangles in the following diagram commute.
    \[\begin{tikzpicture}
         \matrix (m) [matrix of math nodes,row sep=2em,column sep=4em,minimum width=2em]
         {
          D(A) & L \\
          M & D(B) \\};
         \path[->]
         (m-1-1) edge node [left] {$m_A^\dag$} (m-2-1)
                edge node [above] {$l_A^\dag$} (m-1-2)
         (m-1-2) edge node [right] {$l_B$} (m-2-2)
                 edge node [above] {$f$} (m-2-1)
         (m-2-1) edge node [below] {$m_B$} (m-2-2);
    \end{tikzpicture}\]

    $(iii)\Rightarrow (i)$ To prove that $f\colon L\to M$ is unitary, it suffices to establish $f\circ f^\dag=\id[M]$. By the two universal properties of $M$, we may further reduce to pre- and postcomposing with structure maps to and from the diagram $D$. The following diagram commutes.
        \[
        \begin{tikzpicture}
         \matrix (m) [matrix of math nodes,row sep=2em,column sep=4em,minimum width=2em]
         {
          M & L & M \\
          D(A) & M & D(B) \\};
         \path[->]
         (m-1-1) edge node [above] {$f^\dag$} (m-1-2)
         (m-1-2) edge node [above] {$f$} (m-1-3)
                edge node [left=1mm] {$l_B$} (m-2-3)
         (m-1-3) edge node [right] {$m_B$} (m-2-3)
         (m-2-1) edge node [below] {$m^\dag_A$} (m-2-2)
                edge node [left] {$m^\dag_A$} (m-1-1)
                edge node [right=3mm] {$l^\dag_A$} (m-1-2)
         (m-2-2) edge node [below] {$m_B$} (m-2-3);
        \end{tikzpicture}
        \]
    Hence $f$ is unitary.
\end{proof}

The previous theorem highlights why one would want limits in dagger categories to be defined up to unitary isomorphism: if $(L,\{l_A \}_{A\in \cat{J}})$ is a limit of $D$, then $(L,\{l_A^\dag \}_{A\in \cat{J}})$ is a colimit of $\dag\circ D$, and being defined up to unitary iso ensures that the limit and the colimit structures are compatible with each other.

We now set out to prove Theorem~\ref{thm:daglimsunique}. The following lemma will be crucial.

\begin{lemma}\label{lem:connectingmaps}
  Let $l_A \colon L \to D(A)$ and $m_A \colon M \to E(A)$ be dagger limits of $(D,\Omega)$ and $(E,\Omega)$, respectively, where $D,E\colon \cat{J}\rightrightarrows \cat{C}$. Let $\sigma\colon D\to E$ be an adjointable natural transformation. The following diagram commutes for every $A$ and $B$ in $\cat{J}$.
  \begin{equation}\label{diag:connectingmaps}\begin{tikzpicture}
    \matrix (m) [matrix of math nodes,row sep=2em,column sep=4em,minimum width=2em]
    {
     D(A) & L & D(B) \\
     E(A) & M & E(B) \\};
    \path[->]
    (m-1-1) edge node [left] {$\sigma_A$} (m-2-1)
           edge node [above] {$l_A^\dag$} (m-1-2)
    (m-1-2) edge node [above] {$l_B$} (m-1-3)
    (m-1-3) edge node [right] {$\sigma_B$} (m-2-3)
    (m-2-1) edge node [below] {$m_A^\dag$} (m-2-2)
    (m-2-2) edge node [below] {$m_B$} (m-2-3);
  \end{tikzpicture}\end{equation}
\end{lemma}
\begin{proof}
  First we show that it is enough to check that~\eqref{diag:connectingmaps} commutes whenever $A,B\in\Omega$. So assume that it does and let $X,Y\in\cat{J}$ be arbitrary. By weak initiality of $\Omega$ we can find maps $f\colon A\to X$ and $g\colon B\to Y$ for some $A,B\in\Omega$. Now we prove the claim for $X$ and $Y$ assuming the claim for $A$ and $B$ by showing that the diagram 
    \[\begin{tikzpicture}
    \matrix (m) [matrix of math nodes,row sep=2em,column sep=4em,minimum width=2em]
    {
     D(X) & E(X) &  \\
    &  D(A) & E(A) & M \\
    & D(B) & E(B)\\
    L & &D(Y) & E(Y)\\};
    \path[->]
    (m-1-1) edge node [below] {$Df^\dag\ \ $} (m-2-2)
           edge node [above] {$\sigma_X$} (m-1-2)
           edge node [left] {$l_X^\dag$} (m-4-1)
    (m-1-2) edge node [above] {\quad$Ef^\dag$} (m-2-3)
             edge[out=0,in=135] node [above] {$m_X^\dag$} (m-2-4)
    (m-2-2) edge node [below] {$\sigma_A$} (m-2-3)
            edge node [above] {$l_A^\dag\ \ $} (m-4-1)
    (m-2-3) edge node [above] {$m_A^\dag$} (m-2-4)
    (m-2-4) edge node [right] {$m_Y$} (m-4-4)
            edge node [below] {$\ \ m_B$} (m-3-3)
    (m-3-2) edge node [above] {$\sigma_B$} (m-3-3)
            edge node [below] {$Dg$} (m-4-3) 
    (m-3-3) edge node [above] {$Eg$} (m-4-4)
    (m-4-1) edge node [below] {$l_B$} (m-3-2)
            edge node [below] {$l_Y$} (m-4-3)
    (m-4-3) edge node [below] {$\sigma_Y$} (m-4-4);
  \end{tikzpicture}\]
  commutes. The parallelogram on the bottom (top) commutes because $\sigma$ is natural (and adjointable). The triangles on the left of these parallelograms commute because $L$ is a cone and the triangles on the right commute because $M$ is. Finally, the remaining shape is just diagram~\eqref{diag:connectingmaps} with $A,B\in\Omega$.

  For the rest of the proof, write $l_{A,B}:=l_Bl_A^\dag\colon D(A)\to L\to D(B)$ and similarly $m_{A,B}:= m_B m_A^\dag$. Note that $l_{B,A}^\dag=l_{A,B}$ and $m_{B,A}^\dag=m_{A,B}$. Moreover, $l_{A,B}$ and $m_{A,B}$ are partial isometries whenever $A,B\in\Omega$ by Remark~\ref{rem:pisetc}. Our goal is to show that
  \begin{equation}\label{eq:connectingmaps}m_{A,B}\sigma_A=\sigma_B l_{A,B}\end{equation}
  for all $A$ and $B$ in $\cat{J}$. 

  Let $f\colon L\to M$ be the unique map making this square commute for all $A$ in $\cat{J}$:
  \[\begin{tikzpicture}
         \matrix (m) [matrix of math nodes,row sep=2em,column sep=4em,minimum width=2em]
         {
          L & D(A) \\
          M & E(A) \\};
         \path[->]
         (m-1-1) edge[dashed] node [left] {$f$} (m-2-1)
                edge node [above] {$l_A$} (m-1-2)
         (m-1-2) edge node [right] {$\sigma_A$} (m-2-2)
         (m-2-1) edge node [below] {$m_A$} (m-2-2);
  \end{tikzpicture}\]
  If $A \in \Omega$, then $m_A$ is a partial isometry, so the following diagram commutes:
  \[\begin{tikzpicture}[yscale=2.5,xscale=3]
    \node (a1) at (0,1) {$D(A)$};
    \node (a2) at (1,1) {$L$};
    \node (a3) at (2,1) {$D(A)$};
    \node (a4) at (3,1) {$E(A)$};
    \node (b1) at (1,0) {$D(A)$};
    \node (b2) at (2,.5) {$M$};
    \node (b3) at (2,0) {$E(A)$};
    \node (b4) at (3,0) {$M$};
    \draw[dagger,->] (a1) to node[above]{$l_A^\dag$} (a2);
    \draw[dagger,->] (a2) to node[left]{$l_A$} (b1);
    \draw[->] (a2) to node[below]{$f$} (b2);
    \draw[dagger,->] (a2) to node[above]{$l_A$} (a3);
    \draw[->] (a3) to node[above]{$\sigma_A$} (a4);
    \draw[dagger,->] (b2) to node[below=1mm]{$m_A$} (a4);
    \draw[dagger,->] (b2) to node[right]{$m_A$} (b3);
    \draw[dagger,->] (b3) to node[below]{$m_A^\dag$} (b4);
    \draw[dagger,->] (b4) to node[right]{$m_A$} (a4);
    \draw[->] (b1) to node[below]{$\sigma_A$} (b3);
  \end{tikzpicture}\]
  This shows that 
  $
    \sigma_Al_{A,A}=m_{A,A}\sigma_Al_{A,A}\text.
  $
  Repeating the argument with $\sigma$ replaced by $\sigma^\dag$ gives 
  $
    \sigma_A^\dag m_{A,A}=l_{A,A}\sigma_A^\dag m_{A,A}\text.
  $
  Combining the first equation with the dagger of the second shows that for every $A \in \Omega$:
  \begin{equation}\label{eq:proj lemma} 
    m_{A,A}\sigma_A=\sigma_A l_{A,A}\text.
  \end{equation}  

  Next, we check that the following diagram commutes:
  \[\begin{tikzpicture}[yscale=1.75,xscale=3]
    \node (a1) at (0,4) {$D(A)$};
    \node (a2) at (1,4) {$E(A)$};
    \node (a3) at (2,4) {$M$};
    \node (a5) at (4,4) {$E(B)$};
    \node (b2) at (1.5,3.5) {$M$};
    \node (b3) at (2,3) {$E(A)$};
    \node (b4) at (3,3.5) {$E(B)$};
    \node (b5) at (4,3) {$M$};
    \node (c1) at (0,2) {$L$};
    \node (c2) at (1,2.5) {$D(A)$};
    \node (c3) at (2,2) {$M$};
    \node (c5) at (4,2) {$E(A)$};
    \node (d1) at (0,1) {$D(B)$};
    \node (d3) at (2,1) {$E(B)$};
    \node (d5) at (4,1) {$M$};
    \draw[->] (a1) to node[above] {$\sigma_A$} (a2);
    \draw[dagger,->] (a2) to node[above] {$m_A^\dag$} (a3);
    \draw[dagger,->] (a3) to node[above] {$m_B$} (a5);
    \draw[dagger,->] (a1) to node[left] {$l_A^\dag$} (c1);
    \draw[dagger,->] (a2) to node[below left=-1mm] {$m_A^\dag$} (b2);
    \draw[dagger,->] (b2) to node[below left=-1mm] {$m_A$} (b3);
    \draw[dagger,->] (b3) to node[right] {$m_A^\dag$} (a3);
    \draw[dagger,->] (a3) to node[below] {$m_B$} (b4);
    \draw[dagger,->] (b4) to node[below]{$m_B^\dag$} (b5);
    \draw[dagger,->] (b5) to node[right]{$m_B$} (a5);
    \draw[dagger,->] (c1) to node[above]{$l_A$} (c2);
    \draw[->] (c2) to node[below]{$\sigma_A$} (b3);
    \draw[dagger,->] (c3) to node[right]{$m_A$} (b3);
    \draw[dagger,->] (c5) to node[right]{$m_A^\dag$} (b5);
    \draw[dagger,->] (c1) to node[left]{$l_B$} (d1);
    \draw[->] (c1) to node[below]{$f$} (c3);
    \draw[dagger,->] (c3) to node[right]{$m_B$} (d3);
    \draw[->] (d1) to node[below]{$\sigma_B$} (d3);
    \draw[dagger,->] (d3) to node[below]{$m_B^\dag$} (d5);
    \draw[dagger,->] (d5) to node[right]{$m_A$} (c5);
    \draw[gray,draw=none] (a1) to node{(iii)} (c2);
    \draw[gray,draw=none] (b2) to node{(i)} (a3);
    \draw[gray,draw=none] (b4) to node{(ii)} (a5);
    \draw[gray,draw=none] (d3) to node{(vi)} (b5);
    \draw[gray,draw=none] (c2) to node{(iv)} (c3);
    \draw[gray,draw=none] (c1) to node{(v)} (d3);
  \end{tikzpicture}\]
  Region (i) commutes since $m_A$ is a partial isometry and region (ii) since $m_B$ is. Region (iii) is equation~\eqref{eq:proj lemma}. Regions (iv) and (v) commute by definition of  $f$.  Finally, the commutativity of projections $m_A^\dag m_A$ and $m_B^\dag m_B$ shows that region (vi) commutes. 

  Commutativity of the outermost rectangle says that
  \begin{equation}\label{eq:lAB1}
        m_{A,B}\sigma_A=m_{A,B}m_{B,A}\sigma_B l_{A,B}
  \end{equation}
  for every $A,B\in\Omega$.   

  Exchanging $M$ and $L$ and replacing $\sigma$ with $\sigma^\dag$ in the above diagram now gives
  $
        l_{A,B}\sigma_A^\dag=l_{A,B}l_{B,A}\sigma_B^\dag m_{A,B}
  $, and applying the dagger on both sides shows
  \begin{equation}\label{eq:lAB2}
        \sigma_A l_{B,A}=m_{B,A}\sigma_B l_{A,B}l_{B,A} 
  \end{equation}
  for every $A,B\in\Omega$. Exchanging $A$ and $B$ gives 
  \begin{equation}\label{eq:lAB3}
        \sigma_B l_{A,B}=m_{A,B}\sigma_A l_{B,A}l_{A,B} 
  \end{equation}
  for every $A,B\in\Omega$. Finally, combine these equations:
    \begin{align*}
    m_{A,B}\sigma_A&=m_{A,B}m_{B,A}\sigma_B l_{A,B}&&\text{ by }\eqref{eq:lAB1} \\ 
    &=m_{A,B}m_{B,A}\sigma_B l_{A,B}l_{B,A} l_{A,B} &&\text{ since }l_{A,B}\text{ is a partial isometry}\\ 
    &=m_{A,B}\sigma_A l_{B,A}l_{A,B}&&\text{ by }\eqref{eq:lAB2} \\ 
    &=\sigma_B l_{A,B} &&\text{ by }\eqref{eq:lAB3}
    \end{align*}
  This completes the proof.
\end{proof}

\begin{corollary}\label{cor:mapoflimsismapofcolims} 
  In the situation of Lemma~\ref{lem:connectingmaps}, the unique morphism $f \colon L\to M$ satisfying
  \[\begin{tikzpicture}
         \matrix (m) [matrix of math nodes,row sep=2em,column sep=4em,minimum width=2em]
         {
         L & M \\
          D(A) & E(A) \\};
         \path[->]
         (m-1-1) edge node [left] {$l_A$} (m-2-1)
                edge[dashed] node [above] {$f$} (m-1-2)
         (m-1-2) edge node [right] {$m_A$} (m-2-2)
         (m-2-1) edge node [below] {$\sigma_A$} (m-2-2);       
  \end{tikzpicture}\]
  also makes the following diagram commute.
  \begin{equation}\label{diag:colims}
  \begin{aligned}\begin{tikzpicture}
         \matrix (m) [matrix of math nodes,row sep=2em,column sep=4em,minimum width=2em]
         {
         L& M \\
          D(A) & E(A) \\};
         \path[->]
         (m-1-1) edge node [above] {$f$} (m-1-2)
         (m-2-2) edge node [right] {$m_A^\dag $} (m-1-2)
         (m-2-1) edge node [below] {$\sigma_A$} (m-2-2)
                 edge node [left] {$l_A^\dag$} (m-1-1);  
  \end{tikzpicture}\end{aligned}\end{equation}
  In other words, the map of limits $L \to M$ induced by $\sigma$ coincides with the map of colimits $L \to M$ induced by $\sigma\colon \dag\circ D\to\dag \circ E$.
\end{corollary}
\begin{proof} 
  To show that $f$ makes diagram~\eqref{diag:colims} commute, it suffices to postcompose with an arbitrary $m_B$ and show that the following diagram commutes.
        \[
        \begin{tikzpicture}
         \matrix (m) [matrix of math nodes,row sep=2em,column sep=4em,minimum width=2em]
         {
          D(A) & E(A) & M \\
          L & D(B) & E(B )\\ 
           & M \\};
         \path[->]
         (m-1-1) edge node [above] {$\sigma_A$} (m-1-2)
                 edge node [left] {$l_A^\dag$} (m-2-1)
         (m-1-2) edge node [above] {$m_A^\dag$} (m-1-3)
         (m-3-2) edge node [below] {$m_B$} (m-2-3)
         (m-1-3) edge node [right] {$m_B$} (m-2-3)
         (m-2-1) edge node [above] {$l_B$} (m-2-2)
                edge node [below] {$f$} (m-3-2)
          (m-2-2) edge node [above] {$\sigma_B$} (m-2-3);
        \end{tikzpicture}
        \]
  The bottom part commutes by definition of $f$, and the top rectangle by Lemma~\ref{lem:connectingmaps}.
\end{proof}

We can now prove Theorem~\ref{thm:daglimsunique}.

\begin{proof}[of Theorem~\ref{thm:daglimsunique}] If the map of cones $L\to M$ is unitary, it is straightforward to verify that $M$ is a dagger limit of $(D,\Omega)$. For the converse, by Theorem~\ref{thm:equivalenttounitary} it suffices to prove that the diagram
    \[\begin{tikzpicture}
     \matrix (m) [matrix of math nodes,row sep=2em,column sep=4em,minimum width=2em]
     {
      D(A) & L \\
      M & D(B) \\};
     \path[->]
     (m-1-1) edge node [left] {$m_A^\dag$} (m-2-1)
            edge node [above] {$l_A^\dag$} (m-1-2)
     (m-1-2) edge node [right] {$l_B$} (m-2-2)
     (m-2-1) edge node [below] {$m_B$} (m-2-2);
    \end{tikzpicture}\]
commutes for any $A$ and $B$ in $\cat{J}$. This follows from Lemma~\ref{lem:connectingmaps} by choosing $\sigma=\id[D]$, which is clearly adjointable.
\end{proof}

Note that the preceding proof only used Lemma~\ref{lem:connectingmaps} in the special case of $\sigma=\id$. This special case is not sufficient for the sequel and indeed the general versions of Lemma~\ref{lem:connectingmaps} and of Corollary~\ref{cor:mapoflimsismapofcolims} are used later on, \eg in Sections~\ref{sec:globaldaglims} and~\ref{sec:commutativity}.

\section{Completeness}\label{sec:completeness}

Perhaps the most obvious definitions of dagger completeness would be ``every small diagram has a dagger limit for some/all possible weakly initial $\Omega$''. However, here dagger category theory deviates from ordinary category theory. Such a definition would be too strong to allow interesting models, as the following theorems show.

\begin{theorem}\label{thm:finitelydagcompleteimpliesindiscrete} 
  If a dagger category has dagger equalizers, dagger pullbacks and finite dagger products, then it must be indiscrete.
\end{theorem} 

\begin{theorem}\label{thm:dagcompleteimpliesindiscrete} 
  If a dagger category has dagger equalizers and infinite dagger products, then it must be indiscrete.
\end{theorem} 

In proving these degeneration theorems, it is useful to isolate some lemmas that derive key properties implied by having dagger equalizers, dagger pullbacks and finite dagger products.

\begin{lemma}\label{lem:dagpbs->everythingispi} 
  The dagger pullback of a morphism $f\colon A\to B$ along $\id[B]$ exists if and only if $f$ is a partial isometry. In particular, if a dagger category has all dagger pullbacks, then every morphism is a partial isometry. 
\end{lemma}
\begin{proof} 
  If $f$ is a partial isometry, then $A$ is a dagger pullback of $f$ along $\id[B]$, with cone $\id[A]$ and $f$.
  For the converse, let $P$ be a dagger pullback of $f$ along $\id[B]$, with cone $p_A \colon P \to A$ and $p_B \colon P \to B$. Since $A$ is also a pullback, there exists a unique isomorphism $g\colon A\to P$ making the diagram 
  \[\begin{tikzpicture}
         \matrix (m) [matrix of math nodes,row sep=2em,column sep=4em,minimum width=2em]
         {
          A \\
          &P & B \\
          &A & B \\};
         \path[->]
         (m-1-1) edge[out=-90,in=180] node [left] {$\id$} (m-3-2)
                 edge node [right=2mm] {$g$} (m-2-2)
                 edge[dagger,out=0] node [above] {$f$} (m-2-3)
         (m-2-2) edge[dagger] node [left] {$p_A$} (m-3-2)
                edge[dagger] node [above] {$p_B$} (m-2-3)
         (m-2-3) edge node [right] {$\id$} (m-3-3)
         (m-3-2) edge[dagger] node [below] {$f$} (m-3-3);
  \end{tikzpicture}\]
  commute. Since $g$ is an isomorphism, this implies that $p_A=g^{-1}$. Because $p_A$ is a partial isometry, both $p_A$ and $g$ are therefore unitary. Now $f=p_Bg$ means that $f$ factors as the composite of a unitary morphism and a partial isometry and hence is a partial isometry itself.
\end{proof}

It follows from the previous lemma that in a dagger category with dagger pullbacks, every subobject is a dagger subobject since every monomorphism is dagger monic.

Recall that a monoid $(M,+,0)$ is \emph{cancellative} when $x+z=y+z$ implies $y=z$.

\begin{lemma}\label{lem:matrixcalculus}
  If a dagger category has finite dagger products, then it is uniquely enriched in commutative monoids and admits a matrix calculus. If it furthermore has dagger equalizers, then addition is cancellative.
\end{lemma}
\begin{proof} 
  Dagger products are in particular biproducts, and it is well known that biproducts make a category uniquely semiadditive; see \eg~\cite[18.4]{mitchell}, \cite[2.4]{vicary2011completeness}, or~\cite[1.1]{heunen:semimodules}.
  For the claim that dagger equalizers make addition cancellative, we refer to~\cite[2.6]{vicary2011completeness}.
\end{proof}

We can now prove the theorems stated in the beginning of this section.

\begin{proof}[of Theorem~\ref{thm:finitelydagcompleteimpliesindiscrete}] 
  Let $f \colon A \to B$; we will prove that $f=0_{A,B}$. By Lemma~\ref{lem:dagpbs->everythingispi}, the tuple $\langle \id, f\rangle \colon A\to A\oplus B$ is a partial isometry. Expanding this fact using the matrix calculus results in the equations $\id+f^\dag f=\id$ and $f=f+f$. Applying cancellativity of addition to the latter equation now gives $f=0$.
\end{proof}

\begin{proof}[of Theorem~\ref{thm:dagcompleteimpliesindiscrete}] 
  Let $f \colon A \to B$; we will prove that $f=0_{A,B}$. 
  Observe that in Lemma~\ref{lem:matrixcalculus}, infinite dagger products induce the ability to add infinitely many parallel morphisms in the same way as binary dagger products enable binary addition: if $f_i\colon A\to B$ is an $I$-indexed family of morphisms and \cat{C} has $\abs{I}$-ary dagger products, the sum $\sum f_i$ can be defined as the composite $\Delta^\dag_B(\bigoplus_i f_i)\Delta_A\colon A\to\bigoplus_i A\to \bigoplus_i B\to B$, where $\Delta_A\colon A\to\bigoplus_{i} A$ is the diagonal map. This implies that the category is enriched in $\Sigma$-monoids~\cite{haghverdiscott:goi}. 
  Hence the following variant of the Eilenberg swindle makes sense.
  \begin{align*} 
    0+(f+f+\cdots)
    &=f+f+\cdots \\
  &=f+(f+f+\cdots)
  \end{align*}
  It now follows from cancellativity (of binary addition) that $f=0$.
\end{proof}

Note that both theorems rest on an interplay between various dagger limits. Indeed, none of the dagger limits in question are problematic on their own: \cat{Hilb} has dagger equalizers and finite dagger products, \cat{Rel} has arbitrary dagger products, and \cat{PInj} has dagger pullbacks. 

Given Theorems~\ref{thm:finitelydagcompleteimpliesindiscrete} and~\ref{thm:dagcompleteimpliesindiscrete}, dagger completeness is not an interesting concept. However, below we define a (finite) completeness condition for dagger categories, and show that it is equivalent to the existence of dagger products, dagger equalizers, and dagger intersections.

\begin{definition}\label{def:based}
  A class $\Omega$ of objects of a category $\cat{J}$ is called a \emph{basis} when every object $B$ allows a unique $A \in \Omega$ making $\cat{J}(A,B)$ non-empty.
  The category $\cat{J}$ is called \emph{based} when there exists a basis, and \emph{finitely based} when there exists a finite basis.
  We say a dagger category \cat{C} has \emph{(finitely) based dagger limits}, or that it is \emph{(finitely) based dagger complete} if for every category \cat{J} with a (finite) basis $\Omega$, and any diagram $D\colon\cat{J}\to\cat{C}$ the dagger limit of $(D,\Omega)$ exists.
\end{definition} 

\begin{example}
  The shape $\bullet \rightrightarrows \bullet$, giving rise to equalizers, is finitely based. Any (finite) discrete category, the shape giving rise to (finite) products, is (finitely) based. Any  indiscrete category is also finitely based. A category is (finitely) based if and only if it is the (finite) disjoint union of categories, each having a weakly initial object. In particular, any dagger category is based and is finitely based iff it has finitely many connected components.
\end{example}

Since both dagger equalizers and arbitrary dagger products exist in a based dagger complete category, all such categories are indiscrete by Theorem~\ref{thm:dagcompleteimpliesindiscrete} and hence the notion is uninteresting. However, being finitely (based) complete need not trivialize the category as seen below in Corollary~\ref{cor:hilbiscomplete}.

\begin{example}
  Finitely based diagrams need not be finite. For example, given any family $f_i\colon A\to B$ of parallel arrows, their joint dagger equalizer is the dagger limit of a finitely based diagram. Similarly, a dagger intersection is the dagger limit of a diagram of dagger monomorphisms $m_i \colon A_i \to B$, \ie a wide (dagger) pullback. Whenever there are at least two monomorphisms, this diagram is not a based category. However, we may also obtain the dagger intersection as the dagger limit of the diagram consisting of themaps $m_im_i^\dag \colon B \to B$, \ie as the limit of the induced projections on $B$. Then we may always take $\Omega$ to be a singleton. Hence being finitely based dagger complete is a stronger requirement than being finitely complete.
\end{example}

\begin{theorem} \label{thm:finitecompleteness}
  A dagger category is finitely based dagger complete if and only if it has dagger equalizers, dagger intersections and finite dagger products.
\end{theorem} 
\begin{proof}
  One direction is obvious, because the shapes of dagger equalizers, dagger intersections, and finite dagger products are all  finitely based. For the converse, assume that a dagger category $\cat{C}$ has dagger equalizers, dagger intersections and finite dagger products. Let $\cat{J}$ have a finite basis $\Omega$. For a given diagram $D\colon\cat{J}\to\cat{C}$, we will construct a dagger limit of $(D,\Omega)$. Fix $A\in\Omega$. For every parallel pair $f,g\colon A\to B$ of morphisms in \cat{J}, pick a dagger equalizer $e_{f,g}\colon E_{f,g}\to D(A)$ of $D(f)$ and $D(g)$. Pick a dagger intersection $m_A\colon L_A\to D(A)$ of all $e_{f,g}$. Pick a dagger product $L$ of $L_A$ for all $A \in \Omega$, and write $p_A \colon L \to L_A$ for the product cone. Given a morphism  $f\colon A\to B$ in $\cat{J}$ with $A\in \Omega$, define $l_A$ as the composite $D(f) m_Ap_A\colon L\to L_A\to D(A)\to D(B)$. Since $\Omega$ is a basis, the choice of $A$ is forced on us. By construction $l_A$ is independent of the choice of $f\colon A\to B$. It is easy to see that $l_A \colon L \to D(A)$ is a limiting cone, so it remains to check the normalization and independence axioms for $\Omega$. Given $A\in\Omega$, 
  \begin{align*}
    l_Al_A^\dag l_A&=m_Ap_A p_A^\dag m_A^\dag m_Ap_A \\
    &= m_Ap_A p_A^\dag p_A &&\text{ since }m_A\text{ is an isometry}\\
    &= m_Ap_A &&\text{ since }p_A\text{ is a partial isometry}\\
     &=l_A\text.
  \end{align*}
  Moreover, since $p_A \colon L \to D(A)$ forms a product, $l_A^\dag l_A l_B^\dag l_B=0_L=l_B^\dag l_Bl_A^\dag l_A$ if $A,B\in \Omega$.
\end{proof} 

\begin{corollary}\label{cor:hilbiscomplete} 
  The categories \cat{FHilb} and \cat{Hilb} are finitely based dagger complete.
\end{corollary}
\begin{proof} 
  Both categories have dagger equalizers and dagger products. To see that dagger intersections exist, note that a dagger monic corresponds to an inclusion of a closed subspace, and hence the intersection can be computed as the set-theoretic intersection of the corresponding closed subspaces, which is still closed. The result now follows from Theorem~\ref{thm:finitecompleteness}.
\end{proof}

The following theorem characterizes dagger categories having dagger-shaped limits similarly.

\begin{theorem}\label{thm:dagshapedcompleteness}
  Let $\kappa$ be a cardinal number. A dagger category has dagger-shaped limits of shapes with at most $\kappa$ many connected components if and only if it has dagger split infima of projections, dagger stabilizers, and dagger products of at most $\kappa$ many objects. 
\end{theorem}
\begin{proof}
  One implication is trivial. For the other, we first reduce to the case of a single connected component. Given $D\colon\cat{J}\to\cat{C}$, with connected components $\cat{J}_i$, assume that we can use dagger limits of projections and dagger stabilizers to build a dagger limit $L_i$ of each restriction $D_i$. Then we can define the dagger limit of $D$ by taking the dagger product of all the $L_i$. For $L$ is a limit of $D$, as each $L_i$ is a limit of $D_i$. Moreover, since $L_i\to D_i(A)$ is a partial isometry, $L_i\to D_i(A)\to L_i$ is the identity. Therefore each $L\to L_i\to D(A)$ is a partial isometry. If $A$ and $B$ are in the same connected component of $\cat{J}$, then $l_A^\dag l_A=l_B^\dag l_B$. If $A$ and $B$ are in different connected components, then $l_Bl_A=0_{DA,DB}$. Hence the independence equations are satisfied, and $L$ indeed is the dagger limit of $D$. 

  Now focus on the case when $\cat{J}$ is connected. Pick any object $A$ of \cat{J}. For every $f,g\colon A\to B$, take the dagger stabilizer $E_{f,g}\to D(A)$ of $D(f)$ and $D(g)$, and let $p_{f,g}$ be the induced projection $D(A)\to E_{f,g}\to D(A)$ on $D(A)$. Let $\mathcal{P}$ be the set of all such projections, and let $l\colon L\to D(A)$ be the limit. For each $B$ in \cat{J}, choose a morphism $f\colon A\to B$, and define $l_B \colon L\to D(B)$ as $D(f)\circ l$. By construction, $l_B$ is independent of the choice of $f$. We claim that this makes $L$ into a dagger limit of $D$. First, to see it is a cone, take an arbitrary $g\colon B\to C$. Then $D(g)l_B=D(g)D(f)l_A=D(gf)l_A=l_C$, as desired. Because $L\to  D(A)$ is a partial isometry, so is each $L\to D(B)$, as in Example~\ref{ex:daggershaped}. Moreover, $l_A^\dag l_A=l_B^\dag l_B$ when $A,B\in\cat{J}$. Finally, to show that $L$ is indeed a limit, take an arbitrary cone $m_B \colon M \to D(B)$. As it is a cone, it is uniquely determined by $m_A$. But $m_A$ must factor through each stabilizer, so $m_A$ is also a cone for $\mathcal{P}$. The universal property of $L$ ensures that $m_A$ factors through $l_A$. 
\end{proof}

\begin{corollary} 
  The dagger category \cat{Rel} has all dagger-shaped limits, and the dagger categories \cat{FinRel} and $\cat{Span}(\cat{FinSet})$ have all dagger-shaped limits with finitely many connected components.
\end{corollary}
\begin{proof} 
  Observe that \cat{Rel} has all (small) dagger products, and that similarly \cat{FinRel} and $\cat{Span}(\cat{FinSet})$ have finite dagger products. Combine Proposition~\ref{prop:variouscatshavedagshapedlims} and Theorem~\ref{thm:dagshapedcompleteness}.
\end{proof}

We end this section by comparing our notion of finite dagger completeness to that of~\cite{vicary2011completeness}. Thus we must restrict to categories enriched in commutative monoids. This simplifies things, as in the following lemma. 

\begin{proposition}\label{prop:limswithzeros} 
  Let $\Omega$ be a basis for \cat{J} and $D\colon\cat{J}\to\cat{C}$ a diagram in a category with zero morphisms. A limit $l_A \colon L \to D(A)$ of $D$ is a dagger limit of $(D,\Omega)$ if and only if:
    \begin{itemize}
      \item $l_A$ is a partial isometry whenever $A\in\Omega$; and
      \item $l_Bl_A^\dag=0_{A,B}$ whenever $A,B\in\Omega$ and $A\neq B$.
    \end{itemize}
\end{proposition}
\begin{proof} 
  Any limit satisfying the above conditions also satisfies normalization. For independence, the second condition gives $l^\dag_A l_A l_B^\dag l_B=0_{L,L}=l^\dag_Bl_Bl^\dag_Al_A$. 

  Conversely, given a dagger limit $l_A \colon L \to D(A)$ of $(D,\Omega)$, we wish to show that $l_Bl_A^\dag=0_{A,B}$ whenever $A,B\in\Omega$ are distinct. Fix distinct $A,B\in\Omega$ and define a natural transformation $\sigma\colon D\Rightarrow D$ by
  \[
    \sigma_C = \begin{cases} \id[D(C)] & \text{if }\cat{J}(A,C)\neq \emptyset \\
             0_{D(C),D(C)} &\text{otherwise.} \end{cases}
  \]
  By construction $\sigma$ is adjointable. Hence Lemma~\ref{lem:connectingmaps} guarantees $l_Bl_A^\dag=l_Bl_A^\dag\id=l_Bl_A^\dag\sigma_A=\sigma_B l_Bl_A^\dag=0l_Bl_A^\dag=0$.
\end{proof}

\begin{proposition}\label{prop:limswithsums} 
  Let $\Omega$ be a finite basis for \cat{J} and $D\colon\cat{J}\to\cat{C}$ a diagram in a dagger category $\cat{C}$ enriched in commutative monoids. A limit $l_A \colon L \to D(A)$ is a dagger limit of $(D,\Omega)$ if and only if $\id[L]=\sum_{A\in \Omega} l_A^\dag l_A$.
\end{proposition}
\begin{proof}
  Assume first that $L$ is a dagger limit of $(D,\Omega)$. Since \cat{C} is enriched in commutative monoids, the composition $\cat{C}(A,B) \otimes \cat{C}(B,C) \to \cat{C}(A,C)$ is bilinear by definition of tensor product of commutative monoids, and so the unit morphisms of addition are zero morphisms.
  Therefore $l_Bl_A^\dag=0$ by Proposition~\ref{prop:limswithzeros} when $A,B\in\Omega$ are distinct. Hence 
  \[
    l_B\sum_{A\in \Omega} l_A^\dag l_A=l_Bl_B^\dag l_B+ l_B\sum_{B\neq A\in\Omega} l_Al_A  
    =l_B+0=l_B
  \]
  for every $B\in \Omega$, which implies $\id[L]=\sum_{A\in \Omega} l_A^\dag l_A$.

  For the converse, assume that $L$ is a limit of $D$ satisfying $\id[L]=\sum_{A\in \Omega} l_A^\dag l_A$. Fix $A\in\Omega$ and define $\sigma\colon D\Rightarrow D$ as in the proof of Proposition~\ref{prop:limswithzeros}.
  Now $\sigma$ induces a unique morphism $f\colon L\to L$ making the following square commute for any $C$ in $\cat{J}$:
  \[\begin{tikzpicture}
    \matrix (m) [matrix of math nodes,row sep=2em,column sep=4em,minimum width=2em]
    {
    L & L \\
     D(C) & D(C) \\};
    \path[->]
    (m-1-1) edge node [left] {$l_C$} (m-2-1)
           edge node [above] {$f$} (m-1-2)
    (m-1-2) edge node [right] {$l_C$} (m-2-2)
    (m-2-1) edge node [below] {$\sigma_C$} (m-2-2);       
  \end{tikzpicture}\]
  Now
  $
    f
    = \id[L]f 
    = (\sum_{B\in \Omega} l_B^\dag l_B)f 
    = \sum_{B\in \Omega} l_B^\dag l_B f 
    = \sum_{B\in\Omega} l_B^\dag \sigma_B l_B 
    =l_A^\dag l_A +0=l_A^\dag l_A
  $.
  It follows that 
  \begin{equation}\label{eq:limswithsums}
    l_Bl_A^\dag l_A=0 
  \end{equation}
  for distinct $A,B\in\Omega$, so that 
  $
    l_A
    =l_A\id[L] 
    =l_A\sum_{A\in \Omega} l_A^\dag l_A
    =l_A l_A^\dag l_A+0
    =l_A l_A^\dag l_A
  $
  and each $l_A$ is a partial isometry. Finally~\eqref{eq:limswithsums} implies $l_Bl_A^\dag=0_{A,B}$ for distinct $A,B$, making $L$ into a dagger limit of $(D,\Omega)$ by Proposition~\ref{prop:limswithzeros}.
\end{proof}

Thus whenever \cat{J} has a finite basis and \cat{C} is appropriately enriched, Definition~\ref{def:based} coincides with the notion of completeness in~\cite{vicary2011completeness}. But our notion is not more general: when the diagram does not admit a finite basis, the two notions are different. For instance, \cat{FHilb} does not have all dagger pullbacks, because not all morphisms are partial isomorphisms. But for every diagram $A\to C\leftarrow B$ there exists a pullback $A\leftarrow L\to B$ with $\id[L]=l_A^\dag l_A+l_B^\dag l_B$. Conversely, the following example below exhibits a dagger pullback in \cat{Rel} that does not satisfy such a summation.

\begin{example}\label{ex:pullback} 
  Consider the objects $A=\{1,2\}$ and $B=\{1\}$ in $\cat{Rel}$, and morphisms $R=A \times A \colon A \to A$ and $S=B \times A \colon B \to A$.
  Then the object $A$ with the cone morphisms $l_A=\id[A]\colon A\to A$ and $l_B=A\times B\colon A\to B$ form a pullback of $R$ and $S$. Now both $l_A$ and $l_B$ are partial isometries, and $l_A^\dag l_A l_B^\dag l_B = l_B^\dag l_B l_A^\dag l_A$, so this is a dagger pullback.
  But $l_A^\dag l_A+l_B^\dag l_B=A\times A\neq\id[A]$.
\end{example}

\section{Global dagger limits}\label{sec:globaldaglims}

If limits happen to exist for all diagrams of a fixed shape, it is well known that they can also be formulated as an adjoint to the constant functor. This section explores this phenomenon of global limits in the dagger setting.

More precisely, fix a shape \cat{J} and a weakly initial $\Omega$. Assume that a dagger category \cat{C} has $(\cat{J},\Omega)$-shaped dagger limits. Then there is an induced right adjoint $L$ to the diagonal functor $\Delta\colon \cat{C}\to \cat{Cat}(\cat{J},\cat{C})$. 
It restricts to a dagger functor $\hat{L}\colon[\cat{J},\cat{C}]_\dag\to\cat{C}$: $L(\sigma^\dag) = L(\sigma)^\dag$ for any $D,E\colon\cat{J}\to\cat{C}$ and adjointable $\sigma \colon D \Rightarrow E$. This follows from Corollary~\ref{cor:mapoflimsismapofcolims}, since $L(\sigma^\dag)$ is the map of limits induced by $\sigma^\dag\colon E\to D$, and  $L(\sigma)^\dag$ is the map of colimits induced by $\sigma^\dag\colon \dag \circ E\to \dag \circ D$. 
Normalization and independence endow the adjunction $\Delta \dashv L$ with further properties. 

\begin{description}
  \item[Normalization] 
  Consider the counit $\varepsilon\colon \Delta L\to \id[{\cat{Cat}(\cat{J},\cat{C})}]$. By construction, for any diagram $D\colon\cat{J}\to\cat{C}$, each $(\varepsilon_D)_A$ is a partial isometry whenever $A\in\Omega$. ``The counit is a partial isometry when restricted to $\Omega$''. 

  \item[Independence]
  Write $d_A$ for the component of $\varepsilon_D\colon \Delta L(D)\to D$ at $A$, and $e_A$ for the component of $\varepsilon_E$ at $A$. For a fixed $A\in \Omega$, define a natural transformation $\rho_A\colon \hat{L}\to\hat{L}$ by setting $(\rho_A)_D=d_A^\dag d_A\colon L(D)\to D(A)\to L(D)$. To see that this forms a natural transformation $\hat{L}\Rightarrow \hat{L}$, let $\sigma\colon D\Rightarrow E$ be adjointable and consider the following diagram.
  \[\begin{tikzpicture}
           \matrix (m) [matrix of math nodes,row sep=2em,column sep=4em,minimum width=2em]
           {
            L(D) & D(A) & L(D) \\
            L(E) & E(A) & L(E) \\};
           \path[->]
           (m-1-1) edge node [above] {$d_B$} (m-1-2)
                   edge node [left] {$L(\sigma)$} (m-2-1)
           (m-1-2) edge node [above] {$d_A^\dag$} (m-1-3)
                  edge node [left] {$\sigma_A$} (m-2-2)
           (m-1-3) edge node [right] {$L(\sigma)$} (m-2-3)
           (m-2-1) edge node [below] {$e_A$} (m-2-2)
            (m-2-2) edge node [below] {$e_A^\dag$} (m-2-3);
  \end{tikzpicture}\]
The left square commutes by definition of $L(\sigma)$, and the right one by definition of $L(\sigma^\dag)^\dag=L(\sigma)$, so that the rectangle expressing naturality of $\rho_A$ commutes. Clearly, $\rho_A\rho_B=\rho_B\rho_A$ whenever $A,B\in\Omega$.
\end{description}

These properties in fact characterize dagger categories $\cat{C}$ having all $(\cat{J},\Omega)$-shaped limits.

\begin{theorem}\label{thm:globallimits}
  A dagger category \cat{C} has all $(\cat{J},\Omega)$-shaped limits if and only if the diagonal functor $\Delta\colon \cat{C}\to \cat{Cat}(\cat{J},\cat{C})$ has a right adjoint $L$ such that:
    \begin{itemize}
        \item the counit is a partial isometry when restricted to $\Omega$;
        \item $L$ restricts to a dagger functor $\hat{L}\colon[\cat{J},\cat{C}]_\dag\to\cat{C}$;
        \item the family $\rho(A)_D=d_A^\dag d_A$ is natural $\hat{L}\to \hat{L}$ for any $A\in\Omega$;
        \item if $A,B \in \Omega$, then $\rho_A\rho_B=\rho_B\rho_A$.
    \end{itemize}
\end{theorem}
\begin{proof}
  We already proved the implication from left to right above. The other implication is straightforward: it is well-known that if $L$ is the right adjoint to the diagonal then $LD$ is a limit of $D$ for each $D$. Hence we only need to check normalization and independence on $\Omega$. The counit being a partial isometry when restricted to $\Omega$ means that for each $A\in\Omega$ the structure map $LD\to DA$ of the limit is a partial isometry, giving us normalization. The condition $\rho_A\rho_B=\rho_B\rho_A$ for $A,B\in\Omega$ then amounts to independence.
\end{proof}

We leave open the question whether the fourth condition of Theorem~\ref{thm:globallimits} is necessary in general.

Further restricting the shape \cat{J} yields a cleaner special cases of the previous theorem.

\begin{theorem} 
  Let $\Omega$ be a basis for \cat{J}. A dagger category \cat{C} has all $(\cat{J},\Omega)$-shaped limits if and only if the diagonal functor $\Delta\colon \cat{C}\to \cat{Cat}(\cat{J},\cat{C})$ has a right adjoint $L$ such that:
    \begin{itemize}
        \item the counit is a partial isometry when restricted to $\Omega$;
        \item $L$ restricts to a dagger functor $\hat{L}\colon[\cat{J},\cat{C}]_\dag\to\cat{C}$.
    \end{itemize}
\end{theorem} 
\begin{proof} 
    The implication from left to right follows from the previous theorem.
    For the other direction, the conditions imply that a diagram $D$ has a limit given by $d_A \colon L(D) \to D(A)$, and furthermore $d_A$ is a partial isometry if $A\in \Omega$. It remains to verify that $d_A^\dag d_A$ commutes with $d_B^\dag d_B$ when $A,B\in\Omega$. We may assume that $\Omega$ has at least two distinct objects. We will show that \cat{C} has zero morphisms and $d_Bd_A^\dag=0_{D(A),D(B)}$ whenever $A,B\in\Omega$ are distinct. 

    Fix distinct $A,B\in\Omega$ and $X_0\in\cat{C}$. For objects $X,Y$ of $\cat{C}$, define $D_{X,Y}\colon\cat{J}\to\cat{C}$ by 
        \begin{equation*}
            D_{X,Y}(C)= \begin{cases} X & \text{if }\cat{J} (A,C)\neq \emptyset\text, \\
                                       Y &  \text{if }\cat{J} (B,C)\neq \emptyset\text, \\
                                       X_0 & \text{otherwise,}
                          \end{cases}
        \end{equation*}
    and mapping morphisms to \id[X], \id[Y] or \id[X_0] as appropriate. Define $0_{X,Y}$ as the composite $d_A^\dag d_B \colon X=D_{X,Y}(A)\to L(D_{X,Y})\to D_{X,Y}(B)=Y$. Let $f\colon Y\to Z$ be arbitrary. This induces a natural transformation $\sigma\colon D_{X,Y}\Rightarrow D_{X,Z}$ such that $\sigma_C$ is $f$ if $\cat{J} (B,C)\neq \emptyset$ and $\sigma_C=\id$ otherwise. It is clearly adjointable, so the following diagram commutes.
    \[\begin{tikzpicture}
         \matrix (m) [matrix of math nodes,row sep=2em,column sep=4em,minimum width=2em]
         {
          D_{X,Y}(A)=X & LD_{X,Y} & D_{X,Y}(B)=Y \\
         D_{X,Z}(A)=X & LD_{X,X} & D_{X,Z}(B)=Z \\};
         \path[->]
         (m-1-1) edge node [left] {$\sigma_A=\id$} (m-2-1)
                edge node [above] {$$} (m-1-2)
         (m-1-2) edge node [above] {$$} (m-1-3)
                edge node [right] {$L\sigma=(L\sigma^\dag)^\dag$} (m-2-2)
         (m-1-3) edge node [right] {$\sigma_B=f$} (m-2-3)
         (m-2-1) edge node [below] {$$} (m-2-2)
         (m-2-2) edge node [below] {$$} (m-2-3);
    \end{tikzpicture}\]
    Hence $f\circ 0_{X,Y}=0_{X,Z}$. Taking daggers shows that $0_{Y,X}f^\dag=0_{Z,Y}$ for all $X,Y$ and $f\colon Y\to Z$. Thus \cat{C} indeed has zero morphisms.

    Finally, given an arbitrary $D\colon\cat{J}\to\cat{D}$, we will show that $d_{A,B}=0_{D(A),D(B)}$ for distinct $A,B$. Define $\sigma\colon D\Rightarrow D$ by 
    \begin{equation*}
            \sigma_C = \begin{cases} 0_{D(C),D(C)} & \text{if }\cat{J} (A,C)\neq \emptyset\text, \\
                                       \id & \text{otherwise.}
                          \end{cases}
    \end{equation*}
    As $\sigma$ is adjointable, the following diagram commutes.
    \[\begin{tikzpicture}
         \matrix (m) [matrix of math nodes,row sep=2em,column sep=6em,minimum width=2em]
         {
          D(A) & LD & D(B) \\
         D(A)=X & LD & D(B) \\};
         \path[->]
         (m-1-1) edge node [left] {$\sigma_A=0$} (m-2-1)
                edge node [above] {$d_A^\dag $} (m-1-2)
         (m-1-2) edge node [above] {$d_B$} (m-1-3)
                edge node [right] {$L\sigma=(L(\sigma)^\dag)^\dag$} (m-2-2)
         (m-1-3) edge node [right] {$\sigma_B=\id $} (m-2-3)
         (m-2-1) edge node [below] {$d_A^\dag$} (m-2-2)
         (m-2-2) edge node [below] {$d_B$} (m-2-3);
    \end{tikzpicture}\]
    Thus $d_{A,B}=0_{D(A),D(B)}$.
\end{proof}

Recall that a \emph{dagger adjunction} is an adjunction between dagger categories where both functors are dagger functors. There is no distinction between left and right dagger adjoints.

\begin{theorem}
  Let \cat{J} be a dagger category. A dagger category \cat{C} has a dagger limit for every dagger functor $\cat{J}\to \cat{C}$ if and only if the diagonal functor $\Delta\colon \cat{C}\to \cat{Dagcat}(\cat{J},\cat{C})$ has a dagger adjoint such that the counit is a partial isometry.
\end{theorem}
\begin{proof} 
  The implication from left to right is straightforward once one remembers from Example~\ref{ex:daggershaped} that for any dagger limit $(L,\{l_A\}_{A\in\cat{J}})$ of a dagger-shaped diagram, every $l_A$ is a partial isometry.

  For the other direction, the dagger adjoint to the diagonal clearly gives a limit for each dagger functor $\cat{J}\to \cat{C}$. It remains to verify that they are all dagger limits. The counit being a partial isometry implies the normalization condition, so it suffices to check independence. If \cat{J} is connected, this is trivial. If \cat{J} is not connected, then, as in the previous proof, \cat{C} has zero morphisms and for any diagram $D$ we have $d_{A,B}=0_{D(A),D(B)}$ whenever $A$ and $B$ are in different components of \cat{J}.
\end{proof}

\section{Dagger adjoint functors}\label{sec:daft}

When dealing with dagger limits and dagger functors preserving them, the obvious question arises when there exists a dagger adjoint. That is, is there a dagger version of the adjoint functor theorem? 

Simply replacing limits with dagger limits in any standard proof of the adjoint functor theorem~\cite{mitchell} doesn't quite get there. If $\cat{C}$ has all dagger products and dagger equalizers, it must be indiscrete by Theorem~\ref{thm:dagcompleteimpliesindiscrete}. Hence any continuous functor $\cat{C} \to \cat{D}$ satisfying the solution set condition vacuously has an adjoint. A more interesting dagger adjoint functor theorem must therefore work with a finitely based complete category $\cat{C}$ and solution sets of a finite character.

There is a further obstacle. Ordinarily, the adjoint functor theorem shows that its assumptions imply a universal arrow $\eta_A \colon A \to GF(A)$ for each object $A$, so that the desired adjoint is given by $A \mapsto F(A)$. This will not do for dagger categories, as the resulting functor $F$ need not preserve the dagger. This is essentially because universal arrows need not be defined up to unitary isomorphism.
For an example, consider the identity functor $G \colon \cat{FHilb}\to\cat{FHilb}$, and define $\eta_A\colon A\to A$ to be multiplication by $1+\dim A$. Then each $\eta_A$ is an universal arrow for $G$, defining an adjoint $F\colon \cat{FHilb}\to\cat{FHilb}$ that sends $f\colon A\to B$ to $f(1+\dim B)/(1+\dim A)$, which is not a dagger functor. Of course, this example evaporates by choosing the `correct' universal arrows. But there are more involved examples of dagger functors admitting adjoints but no dagger adjoints. 

The moral is that a dagger adjoint to $G\colon \cat{C}\to\cat{D}$ requires more than $G$-universal arrows $A\to GF(A)$ for each object. The universal arrows must fit together, in the sense that they form an adjointable natural transformation. 
Unfortunately, we do not know of conditions on solution sets guaranteeing this. 
As a first step towards a dagger adjoint functor theorem proper we provide the following theorem.

\begin{theorem}\label{thm:adjunction}
  Suppose a dagger category \cat{C} has, and a dagger functor $G \colon \cat{C} \to \cat{D}$ preserves, dagger intersections and dagger equalizers. Then $G$ has a dagger adjoint if and only if there is a dagger functor $H\colon \cat{D}\to\cat{C}$ and a natural transformation $\tau \colon \id[\cat{D}]\to GH$ such that each component of $\tau$ is weakly G-universal.
\end{theorem}
\begin{proof} 
  One implication is trivial. For the other, define $F(A)$ to be a dagger intersection of all dagger monomorphisms $m\colon M\to H(A)$ for which $\tau_A$ factorizes through $G(m)$. As $G$ preserves dagger intersections, $\tau_A$ factorizes via $GF(A)\to GH(A)$, say $\tau_A=G(\sigma_A)\eta_A$. It suffices to prove that (i) $\eta_A$ is $G$-universal, so that $A\mapsto FA$ extends uniquely to a functor, and then that (ii) $F$ is a dagger functor. 

  First of all, $\eta_A$ is weakly $G$-universal since $\tau_{A}$ is so: given $f\colon A\to G(X)$, pick $h$ such that $f=G(h)\tau_A=G(h\sigma_A)\eta_A$. Moreover, if $h$ and $h'$ satisfied $G(h)\eta_A=G(h')\eta_A$, consider the equalizer $e\colon E\to F(A)$ of $h$ and $h'$. By assumption, $\eta_A$ factors through $G(e)$ and hence $\tau_A$ factors through $G(\sigma_A e)$, so that $\sigma_A e$ is already in the dagger intersection defining $F(A)$. In other words, $e$ is unitary, and hence $h=h'$. Thus $\eta_A$ is $G$-universal, and we can extend $A\mapsto F(A)$ to a functor by defining $F(f)$ as the unique map making the following square commute.
  \[\begin{tikzpicture}
    \matrix (m) [matrix of math nodes,row sep=2em,column sep=4em,minimum width=2em]
    { A & GFA \\
      B & GFB \\};
    \path[->]
    (m-1-1) edge node [left] {$f$} (m-2-1)
            edge node [above] {$\eta_A$} (m-1-2)
    (m-1-2) edge node [right] {$GF(f)$} (m-2-2)
    (m-2-1) edge node [below] {$\eta_B$} (m-2-2); 
  \end{tikzpicture}\]

  Next we show that $F$ is a dagger functor. By Lemma~\ref{lem:daggersubfunctorsaredagger} it suffices to show that $\sigma$ is a natural transformation $F\to H$. By naturality of $\tau$, the top part of
  \[\begin{tikzpicture}
    \matrix (m) [matrix of math nodes,row sep=2em,column sep=4em,minimum width=2em]
    { GFA &&& GHA \\
      A & B & GFB & GHB \\
      & GFA \\};
    \path[->]
    (m-1-1) edge node [above] {$G(\sigma_A)$} (m-1-4)
    (m-1-4) edge node [right] {$GH(f)$} (m-2-4)
    (m-2-1) edge node [above] {$f$} (m-2-2)
        edge node [left] {$\eta_A$} (m-1-1)
          edge node [below] {$\eta_A$} (m-3-2)
    (m-3-2) edge node [below] {$\quad GFf$} (m-2-3)
    (m-2-3) edge node [above] {$G(\sigma_B)$} (m-2-4)
    (m-2-2) edge node [above] {$\eta_B$} (m-2-3);
  \end{tikzpicture}\]
  commutes, whereas the bottom part commutes by naturality of $\eta$. As $\eta_A$ is $G$-universal, we conclude that the square 
  \[\begin{tikzpicture}
    \matrix (m) [matrix of math nodes,row sep=2em,column sep=4em,minimum width=2em]
    { FA & HA \\
    FB & HB \\};
    \path[->]
    (m-1-1) edge node [left] {$F(f)$} (m-2-1)
            edge node [above] {$\sigma_A$} (m-1-2)
    (m-1-2) edge node [right] {$H(f)$} (m-2-2)
    (m-2-1) edge node [below] {$\sigma_B$} (m-2-2); 
  \end{tikzpicture}\]
  commutes, making $\sigma \colon F \Rightarrow H$ natural.
\end{proof}

\section{Polar decomposition}\label{sec:polar}

Polar decomposition, as standardly understood, provides a way to factor any bounded linear map between Hilbert spaces into a partial isometry and a positive morphism~\cite[Chapter 16]{halmos:problembook}. As dagger limits are defined in terms of partial isometries, one might hope that polar decomposition connects dagger limits and ordinary limits. This section explores this connection. We start by defining polar decomposition abstractly, and prove that this modified property holds in the category of Hilbert spaces. Recall that a bimorphism in a category is a morphism that is both monic and epic.

\begin{definition}\label{def:polar}
  Let $f\colon A\to B$ be a morphism in a dagger category. A \emph{polar decomposition} of $f$ consists of two factorizations of $f$ as $f=pi=jp$,
  \[\begin{tikzpicture}
     \matrix (m) [matrix of math nodes,row sep=2em,column sep=4em,minimum width=2em]
     {
      A & A \\
      B & B \\};
     \path[->]
     (m-1-1) edge[dagger] node [left] {$p$} (m-2-1)
            edge node [above] {$i$} (m-1-2)
            edge node [right=3mm] {$f$} (m-2-2)
     (m-1-2) edge[dagger] node [right] {$p$} (m-2-2)
     (m-2-1) edge node [below] {$j$} (m-2-2);
  \end{tikzpicture}\]
  where $p$ is a partial isometry and $i$ and $j$ are self-adjoint bimorphisms. 
  A category \emph{admits polar decomposition} when every morphism has a polar decomposition.
\end{definition}

Note that if $f$ is monic then so is $p$: indeed, if $pg=ph$ then $fg=jpg=jph=fh$, whence $g=h$. Hence if $f$ is monic then $p$ is dagger monic since it is both monic and a partial isometry.

Unlike usual expositions of polar decomposition, we require $i$ and $j$ to be bimorphisms. On the other hand, we don't require them to be positive since mere self-adjointness suffices for our purposes -- most notably Theorems~\ref{thm:polarlimitsaredaggerlimits} and~\ref{thm:daglimsofisomorphicdiagrams}. However, for other purposes this definition might need to be modified and hence it should be only seen as a starting point for an abstract notion of polar decomposition.

We begin by proving that \cat{Hilb} admits polar decomposition in the above sense. Given that our notion of polar decomposition is slightly different from the usual one (merely, a single factorization as $f=pi$ with $p$ a partial isometry and $i$ positive), there is some work to do. First, recall that if $f$ is a bounded linear map, i.e. a morphism of \cat{Hilb}, then $\abs{f}$ is the (unique) positive square root of $f^\dag f$. 

\begin{lemma}\label{lem:absolutevaluecommutes} For any morphism $f$ in \cat{Hilb} we have $\abs{f^\dag}f=f\abs{f}$.
\end{lemma}
\begin{proof} First we show that it suffices to prove the claim for non-expansive $f$, \ie  we may assume that $\norm{f}\leq 1$. Assuming the claim holds for non-expansive maps, take $f$ with $\norm{f}>1$. Define $g:=f/\norm{f}$. Now $\abs{g}=\abs{f}/\norm{f}$ since positive square roots are unique and both sides square to $f^\dag f/\norm{f}^2$. Hence $\abs{g^\dag}g=g\abs{g}$ amounts to saying that 
    \[\abs{f^\dag}f/\norm{f}^2=f\abs{f}/\norm{f}^2\]
  so that multiplying by $\norm{f}^2$ gives the result for $f$. Hence we may assume that $f$ is non-expansive.

  By (the proof of) \cite[Theorem 23.2]{bachman2012functional}, the square root of a positive non-expansive operator $p$ is the strong limit of the sequence $(p_n)_{n\in\N}$, where $p_0=0$ and $p_{n+1}:=p_n-(p-p_n^2)/2$. Applying this result to $f^\dag f$ we see that $\abs{f}$ is the strong limit of the sequence $(q_n)_{n\in\N}$ where $q_0=0$ and $q_{n+1}:=q_n-(f^\dag f-q_n^2)/2$. Similarly, $\abs{f^\dag}$ is the strong limit of the sequence $(r_n)_{n\in\N}$ defined by $r_0=0$ and $r_{n+1}:=r_n-(ff^\dag-r_n^2)/2$. Since precomposing (and postcomposing) with a non-expansive map is continuous (in the strong operator topology), to prove that $\abs{f^\dag}f=f\abs{f}$ it suffices to prove that $r_nf=fq_n$, which we do by induction. For $n=0$, both sides evaluate to $0$. Assuming the claim holds for $n$ we see that it holds for $n+1$, since
  \begin{align*}
    r_{n+1} f&=(r_n-(ff^\dag-r_n^2)/2)f \\
            &=fq_n-(ff^\dag f-fq_n^2)/2\text{ by the induction hypothesis} \\
            &=f(q_n-(f^\dag f-q_n^2)/2)=fq_{n+1}
  \end{align*}
  completing the proof.
\end{proof}

\begin{theorem}
  The category $\cat{Hilb}$ admits polar decomposition. 
\end{theorem}
\begin{proof}
  We modify the standard construction of a polar decomposition~\cite[Problem 134]{halmos:problembook} to satisfy Definition~\ref{def:polar}: the standard construction gives a factorization of $f\colon H\to K$ into $f=p\abs{f}$, where $p$ is a partial isometry satisfying $\ker p=\ker f$. This in fact fixes $p$ uniquely: $H$ decomposes into a direct sum $H\cong \ker f\oplus\ker f^\perp$ and since (i) $\ker f=\ker\abs{f}=\ker p$  and (ii) $\ker \abs{f}^\perp=\overline{\im \abs{f}}$ (see \eg\cite[Lemma 2.1]{buschetal:quantummeasurement}), $H$ in fact decomposes as $H\cong \overline{\im \abs{f}}\oplus \ker p$. The action of $p$ on its kernel is clear and on its orthocomplement it has to be given by continuous extension of $\abs{f}x\mapsto fx$ for $p\abs{f}=f$ to hold. Now, let $r$ be the projection onto $\ker p=\ker\abs{f}$, and set $i:=\abs{f}+r$. On $\ker \abs{f}$ $i$ acts as the identity and on its orthocompletent it acts as $\abs{f}$. Hence $i$ is positive. By construction, $\ker i=0$ so that $i$ is monic. Being self-adjoint it is also epic. Moreover, since $\ker p=\ker\abs{f}$ the factorization $f=pi$ follows from $f=p\abs{f}$.

  Similarly $f^\dag$ factors as $f^\dag=q\abs{f^\dag}$ where $q$ is defined similarly, and this factorization can be modified to obtain $f^\dag=q j$, where $j$ is a self-adjoint (in fact positive) bimorphism. Taking daggers, we have $f=jq$. 

  Hence it remains to show that $q^\dag=p$. Since the dagger in \cat{Hilb} is given by adjoints of bounded linear maps, this boils down to showing that
    \begin{equation}\label{eq:inprod}\inprod{px}{y}=\inprod{x}{qy}\end{equation}
  holds for all $x\in H$ and $y\in K$. We can now use our orthogonal decompositions $H\cong \overline{\im \abs{f}}\oplus \ker p$ and $K\cong\overline{\im \abs{f^\dag}}\oplus \ker q $ and consider separately the cases where $x$ and $y$ are in each of the summands. If $x\in \ker p$ then the left hand side of \eqref{eq:inprod} equals zero. But then $x$ is ortohogonal to $\im {f}=\im f=\im q$ so the right hand side is zero as well. Similarly, both sides equal zero when $y\in\ker q$. Thus we're left to consider the case when $x\in \overline{\im \abs{f}}$ and $y\in \overline{\im \abs{f^\dag}}$, and by continuity we may further assume $x\in \im \abs{f}$ and $y\in\im\abs{f^\dag}$. Pick $z\in H$ and $w\in K$ with $\abs{f}z=x$ and $\abs{f^\dag} w=y$. Now $px=p\abs{f}z=fz$ and $qy=q\abs{f^\dag}w=f^\dag w$, so \eqref{eq:inprod} boils down to whether 
    \[\inprod{fz}{\abs{f^\dag}w}=\inprod{\abs{f}z}{f^\dag w}\]
  Now the left hand side equals $\inprod{\abs{f^\dag}fz}{w}$  and the right hand side equals $\inprod{f\abs{f}z}{w}$ so they are equal by Lemma~\ref{lem:absolutevaluecommutes}, completing the proof.
\end{proof}

In \cat{Hilb}, we can not guarantee that $i$ and $j$ are isomorphisms in general. A good example is when $f\colon \ell^2(\N)\to\ell^2(\N)$ is defined on the $n$-th basis element $e_n$ by $f(e_n)=e_n/(n+1)$. Then the factorization above gives $\id$ as the partial isometry and $f$ as the positive bimorphism. However, $f$ does not have an inverse in \cat{Hilb} -- indeed, the ``inverse'' defined by $e_n\mapsto (n+1)e_n$ is not bounded.

Other dagger categories admitting polar decomposition include inverse categories, such as $\cat{PInj}$, and any groupoid, in which every morphism itself is already a partial isometry.

Let us remark here that polar decomposition is quite unlike an orthogonal factorization system. First, the composition of partial isometries need not be a partial isometry, and the composition of self-adjoint bimorphisms need not be self-adjoint. Second, an isomorphism need not be a partial isometry nor self-adjoint. Third, $p$, $i$, and $j$ are not required to be unique to $f$. Fourth, the factorization $f=pi$ respects the dagger: even though one may also factor $f^\dag=qj$ and hence $f=j^\dag q$, we are additionally requiring that $p=q$.

Recall that a dagger category is unitary whenever two objects being isomorphic implies that they are also unitarily isomorphic.

\begin{proposition} 
  Dagger categories that have polar decomposition are unitary. 
\end{proposition}
\begin{proof}
  Factor an isomorphism $f\colon A\to B$ as $f=pi=jp$ with $p$ a partial isometry and $i,j$ self-adjoint bimorphisms.  Now $pi=f$ implies that $p$ has a right inverse and $jp=f$ that it has a left inverse. Hence $p$ must be an isomorphism. Also being a partial isometry, $p \colon A \to B$ is therefore unitary.
\end{proof}

The theme of the rest of this section will be that dagger limits may be viewed as the partial isometry part of a polar decomposition of ordinary limits. 

\begin{example}\label{ex:polarlimitsaredaggerlimits}
  The theme of the rest of this section will be that dagger limits may be viewed as the partial isometry part of a polar decomposition of ordinary limits. Here are some examples to warm up to this theme.
  \begin{itemize}
    \item   
    Let $e \colon E \to A$ be an equalizer of morphisms $f,g \colon A \to B$ in a dagger category. If $e=pi=jp$ is a polar decomposition, then $p \colon A \to B$ is a dagger equalizer of $f$ and $g$.
    Indeed, since $i$ is a bimorphism we see that $fp=gp$, so that $p$ factors through $e$ as $p=ek$ for some $k\colon A\to A$. Precomposing with $i$ we see that $e=eki$, whence $ki=\id$. Since $i$ is a bimorphism and has a left inverse, it is an isomorphism and $k=i^{-1}$. Hence $p=ei^{-1}$ is an equalizer and hence monic, and by definition $p$ is a partial isometry and hence dagger monic.

   \item 
    A cone over a family of dagger monomorphisms $f_i \colon A_i \to B$ consists formally of a map $l_i\colon A\to A_i$ for each $i$ and of a map $f\colon A\to B$. However, the whole cone is determined by the map $f$: this is because $f_il_i=f$ implies $l_i=f_i^\dag f_i l_i=f_i^\dag f$ by dagger monicness of $f_i$. A map $f\colon A\to B$ defines a cone in this manner iff $f_i f_i^\dag f=f$ for each $i$. Now, if $f$ defines a limiting cone and $f=pi=jp$ is a polar decomposition, then $p \colon A \to B$ is a dagger intersection of $\{f_i\}$: since $i$ is a bimorphism we see that $p$ defines a cone for $f_i$ and hence factors through $f$. As in the equalizer case, this implies that $i$ is an isomorphism whence  $p=fi^{-1}$ is a pullback of monics and hence monic. Since $p$ is a partial isometry by definition it is also dagger monic.

     \item 
     Let a projection $e=e^\dag=e^2 \colon A \to A$ be split by $f \colon B \to A$ and $g \colon A \to B$, so $e=fg$ and $gf=\id[B]$. If $f=pi=jp$ is a polar decomposition, then $p \colon B \to A$ is a dagger splitting of $e$: indeed $f$ is the equalizer of $\id$ and $e$ and hence so is $p$ by the above, so that $p$ is monic and hence dagger monic. It remains to check that $e=pp^\dag$. Now $ep=p$ so that $pp^\dag=epp^\dag$ as well. Applying the dagger to both sides of this results in $pp^\dag=pp^\dag e$.  On the other hand, $e=e^2=fge=pige=pp^\dag pige=pp^\dag fge=pp^\dag e^2=pp^\dag e$. Hence $e=pp^\dag e=pp^\dag$.
  \end{itemize}
\end{example}

These examples are no accident, and the theorems proven below will bear out this theme. Before studying polar decompositions of more general diagrams we need the following lemma.

\begin{lemma}\label{lem:PiOfPolarCommutes}
  Let $f\colon A\to B$  be a morphism in a dagger category with polar decomposition $f=pi=jp$, and let $g\colon A\to A$ be self-adjoint. If $g$ commutes with $f^\dag f$, then it also commutes with $p^\dag p$.
\end{lemma}
\begin{proof} 
  Observe that $f^\dag f = p^\dag p i i$.
  \[\begin{tikzpicture}
    \matrix (m) [matrix of math nodes,row sep=2em,column sep=4em,minimum width=2em]
    {
     A & B & A \\
     A & B \\
     A \\};
    \path[->]
    (m-1-1) edge node [left] {$i$} (m-2-1)
           edge node [above] {$f$} (m-1-2)
    (m-1-2) edge node [above] {$f^\dag$} (m-1-3)
            edge node [left] {$j$} (m-2-2)
    (m-2-2) edge[dagger] node [below] {$p^\dag$} (m-1-3)
    (m-2-1) edge[dagger] node [above] {$p$} (m-1-2)
             edge node [below] {$f$} (m-2-2)
             edge node [left] {$i$} (m-3-1)
    (m-3-1) edge[dagger] node [below] {$p$} (m-2-2);
  \end{tikzpicture}\]
  It follows that 
  \[
    p^\dag p g p^\dag p i i
    = p^\dag p g f^\dag f 
    =p^\dag p f^\dag f g 
    =f^\dag f g 
    =g f^\dag f 
    =g p^\dag p i i\text.
  \]   
  Because $i$ is a bimorphism, $p^\dag p g p^\dag p =g p^\dag p$.
  Since $g$ is self-adjoint, the left hand side of this equation is self-adjoint. Therefore also the right-hand side is self-adjoint. Thus $p^\dag p g  =g p^\dag p$.
\end{proof}

The next theorem roughly shows that ``polar decomposition turns based limits into dagger limits''.

\begin{theorem}\label{thm:polarlimitsaredaggerlimits}
  Let $\Omega$ be a basis of \cat{J}. Assume that $D\colon \cat{J}\to\cat{C}$ has a limit $l_A \colon L \to D(A)$ satisfying $l_A^\dag l_A l_B^\dag l_B=l_B^\dag l_Bl_A^\dag l_A$ for all $A,B\in \Omega$. If \cat{C} admits polar decomposition, $(D,\Omega)$ has a dagger limit.
\end{theorem}
\begin{proof} 
  Pick a polar decomposition $l_A = p_Ai_A=j_Ap_A$ for each $A \in \Omega$. For $B\in \cat{J}\setminus{\Omega}$, set $p_B=D(f)p_A$, where $A$ is the unique object in $\Omega$ with $\cat{J}(A,B)\neq\emptyset$. If $f,g\colon A\to B$, then 
  \[
    D(f)p_Ai_A=D(f)l_A=D(g) l_A=D(g)p_Ai_A
  \]
  and hence $D(f)p_A=D(g)p_A$. So $p_B$ is independent of the choice of $f\colon A\to B$, and $p_A \colon L \to D(A)$ forms a cone. By construction, each $p_A$ is a partial isometry whenever $A\in \Omega$. Moreover, by assumption and Lemma~\ref{lem:PiOfPolarCommutes}, $p_A^\dag p_A$ commutes with $l_B^\dag l_B$ when $A,B \in \Omega$. Then, by another application of the lemma $p_A^\dag p_A$ commutes with $p_B^\dag p_B$. 

  It remains to show that $p_A \colon L \to D(A)$ forms a limiting cone. We will establish this by proving that the unique map $f \colon L \to L$ of cones from $p_A$ to $l_A$ is an isomorphism. Thus we need to find a map $g \colon L \to L$ from $l_A$ to $p_A$ that is the inverse of $f$. That $g$ is a map of cones means that the triangle 
  \[\begin{tikzpicture}
    \matrix (m) [matrix of math nodes,row sep=2em,column sep=4em,minimum width=2em]
    {
     L & D(A) \\
     L  \\};
    \path[->]
    (m-1-1) edge node [left] {$g$} (m-2-1)
           edge node [above] {$l_A$} (m-1-2)
    (m-2-1) edge[dagger] node [below] {$p_A$} (m-1-2);
  \end{tikzpicture}\]
  commutes for each $A$ in $\cat{J}$, or equivalently for each $A\in \Omega$. Postcomposing with the bimorphisms $j_A$ we see that this is equivalent to finding $g\colon L\to L$ such that 
  \[\begin{tikzpicture}
    \matrix (m) [matrix of math nodes,row sep=2em,column sep=4em,minimum width=2em]
    {
     L& D(A) \\
     L & D(A) \\};
    \path[->]
    (m-1-1) edge node [left] {$g$} (m-2-1)
           edge node [above] {$l_A$} (m-1-2)
    (m-1-2) edge node [right] {$j_A$} (m-2-2)
    (m-2-1) edge node [below] {$l_A$} (m-2-2);       
  \end{tikzpicture}\]
  for each $A\in \Omega$. As $l_A \colon L \to D(A)$ is a limit cone, the existence of such a $g$ follows as soon as $j_A l_A \colon L \to D(A)$ with $A \in \Omega$ generates a cone. But
  \[
    j_A l_A=j_Ap_Ai_A= l_Ai_A
  \]
  and $l_Ai_A \colon L \to D(A)$ with $A \in \Omega$ obviously generates a cone. Thus we have found a cone map $g \colon L \to L$ from $l_A$ to $p_A$. It suffices to show that it is the inverse of $f$. On the one hand $fg=\id[L]$ by the universal property of the cone $l_A$. On the other hand, $p_A g f = p_A$ for each $A\in \Omega$, and by postcomposing with $j_A$ we see that $gf$ is also a cone map from $l_A$ to $l_A$, and thus equal to the identity.
\end{proof}

In particular, if \cat{C} admits polar decomposition and $A_1$ and $A_2$ have a product $(A_1\times A_2,p_1,p_2)$ satisfying $p_1^\dag p_1 p_2^\dag p_2=p_2^\dag p_2 p_1^\dag p_1$ then the dagger product of $A_1$ and $A_2$ exists as well. Using polar decomposition and splittings of projections, one can also construct the dagger product of $A$ and $B$ from their biproduct, without any further conditions required from the biproduct. 

Recall that, in a category with a zero object (or more generally, zero morphisms), a biproduct of $A_1$ and $A_2$ consists of $(A_1\oplus A_2,p_1,p_2,i_1,i_2)$ where $(A_1\oplus A_2,p_1,p_1)$ is a product of $A_1$ and $A_2$, $(A_1\oplus A_2,i_1,i_2)$ is their coproduct, and moreover
\begin{align*}
  p_1i_1&=\id[A_1] \quad &p_2i_2=\id[2] \\
  p_2i_1&=0_{A_1,A_2} \quad &p_1i_2=0_{A_2,A_1}
\end{align*}
or more succinctly, 
  \[p_ni_k=\delta_{k,n}:=\begin{cases}\id[A_n]\text{ if }n=k \\
                                       0_{A_k,A_n}\text{ otherwise.} \end{cases}\]

\begin{theorem}\label{thm:polarofbiprod} Let \cat{C} be a dagger category that admits polar decomposition and has dagger splittings of projections. If two objects $A_1$ and $A_2$ of \cat{C} have a biproduct, they have a dagger product as well.
\end{theorem}

\begin{proof}
  Given two objects $A_1$ and $A_2$, take their biproduct $(A_1\oplus A_2,p_1,p_2,i_1,i_2)$.
  Assume first that we've produced a cone $q_i\colon P\to A_i$ such that $q_nq^\dag_k=\delta_{n,k}$.  As $(A_1\oplus A_2, i_1,i_2)$ is a coproduct and $(A_1\oplus A_2,p_1,p_2)$ is a product we can find unique maps $f\colon A_1\oplus A_2\to P$ and $g\colon P\to A_1\oplus A_2$ making the diagram 
  \[\begin{tikzpicture}
     \matrix (m) [matrix of math nodes,row sep=2em,column sep=4em,minimum width=2em]
     {
      A_k & & A_n \\
      A_1\oplus A_2 & P & A_1\oplus A_2 \\};
     \path[->]
     (m-1-1) edge node [left] {$i_k$} (m-2-1)
             edge node [above] {$q_k^\dag$} (m-2-2)
            edge node [above] {$\delta_{n,k}$} (m-1-3)
      (m-2-1) edge node [below] {$f$} (m-2-2) 
      (m-2-2) edge node [above] {$q_n$} (m-1-3)
              edge node [below] {$g$} (m-2-3)
      (m-2-3) edge node [right] {$p_n$} (m-1-3);
  \end{tikzpicture}\]
commute. By the universal properties of $A_1\oplus A_2$ this implies that $gf=\id$. Now, $f$ is a map of cocones $(A_1\oplus A_2,i_1,i_2)\to (P,q_1^\dag,q_2^\dag)$. We will prove that it is also a map of cones $(A_1\oplus A_2,p_1,p_2)\to (P,q_1,q_2)$. Consider the diagram 
        \[\begin{tikzpicture}
     \matrix (m) [matrix of math nodes,row sep=2em,column sep=4em,minimum width=2em]
     {
      A_1\oplus A_2 &  & P \\
      A_k &  A_1\oplus A_2  & A_n \\};
     \path[->]
     (m-2-1) edge node [left] {$i_k$} (m-1-1)
            edge node [below] {$i_k$} (m-2-2)
            edge node [above] {$q_k^\dag$} (m-1-3)
     (m-1-1) edge node [above] {$f$} (m-1-3)
     (m-1-3) edge node [right] {$q_n$} (m-2-3)
     (m-2-2) edge node [below] {$p_n$} (m-2-3);
    \end{tikzpicture}\]
The bottom part commutes since both paths equal $\delta_{n,k}$ and the top part commutes by definition of $f$. Hence the rectangle commutes, and since $i_1$ and $i_2$ are jointly epic this shows that $f$ is a map of cones. Now, $g^\dag$ is a map of cocones $(A_1,A_2,p_1^\dag,p_2^\dag)\to (P,q_1^\dag,q_2)$ so the same argument shows that $g^\dag$ is also a map of cones. Taking the dagger again, this means that $g$ is not only a map of cones $(P,q_1,q_2)\to (A_1\oplus A_2,p_1,p_2)$ but also a map of cocones $(P,q_1^\dag,q_2^\dag)\to (A_1\oplus A_2,i_1,i_2)$. These observations are true for any cone $P$ satisfying $q_nq^\dag_k=\delta_{n,k}$. As any such cone satisfies normalization and independence, it is sufficient to find a cone $P$ for which we can prove that $fg=\id[P]$ since then $(P,q_1,q_2)$ will also be a product. Finding such a $P$ is what we'll do in the remainder.

  Factorize $p_1$ as $p_1=pi=jp$ and $i_2$ as $i_2=kr=rl$. Set $d_1:=p$ and $d_2:=r^\dag$. We claim that
    \[d_nd_k^\dag=\delta_{n,k}\]
  Indeed, since $p_1$ and $i_2^\dag$ are epimorphisms the maps $d_1$ and $d_2$ are dagger epic. Moreover, $0=p_2i_1=j d_1d_2^\dag l=j0l$ whence $d_1d_2^\dag=0$ and thus $d_2d_1^\dag=0$. Hence $(A_1\oplus A_2,d_1,d_2)$ is a cone satisfying (i), so as remarked at the end of the first half of the proof, we get maps 
    \[f\colon (A_1\oplus A_2,p_1,p_2,i_1,i_2)\leftrightarrows (A_1\oplus A_2,d_1,d_2,d_1^\dag,d_2^\dag)\colon g\] 
  that are maps of cones and cocones. Taking daggers, we have maps 
    \[g^\dag\colon (A_1\oplus A_2,i_1^\dag,i_2^\dag,p_1^\dag,p_2^\dag)\leftrightarrows (A_1,A_2,d_1,d_2,d_1^\dag,d_2^\dag)\colon f^\dag\]
  that are also compatible with the (co)cone structures. Moreover, since $(A_1\oplus A_2,p_1,p_2,i_1,i_2)$ and $(A_1\oplus A_2,i_1^\dag,i_2^\dag,p_1^\dag,p_2^\dag)$ are biproducts, there is exactly one (co)cone map in both directions, and these are isomorphisms. Since 
    \[f^\dag f\colon (A_1\oplus A_2,p_1,p_2,i_1,i_2)\leftrightarrows(A_1\oplus A_2,i_1^\dag,i_2^\dag,p_1^\dag,p_2^\dag) \colon gg^\dag\]
  are (co)cone maps, these maps have to be inverses to each other. 
  Hence $h:=fg g^\dag f^\dag$ is a self-adjoint cone map $(A_1\oplus A_2,d_1,d_2)\to (A_1\oplus A_2,d_1,d_2)$ satisfying
    \[h^2=(fg g^\dag f^\dag)(fg g^\dag f^\dag)=f(g g^\dag f^\dag f)g g^\dag f^\dag=f g g^\dag f^\dag=h\]
  Thus $h$ is a projection; let $e\colon P\to (A_1\oplus A_2)$ split it and set $q_i:=d_ie$. Now  \[q_nq_k^\dag=d_n ee^\dag d_k^\dag=d_n h d_k^\dag=d_nd_k=\delta_{n,k}\]
  so we have canonical maps 
      \[\tilde{f}\colon (A_1\oplus A_2,p_1,p_2,i_1,i_2)\leftrightarrows (P,q_1,q_2,q_1^\dag,q_2^\dag)\colon \tilde{g}\] 
  and $\tilde{g}\circ\tilde{f}=\id[A_1\oplus A_2]$ holds automatically, whereas whether $\tilde{f}\circ \tilde{g}=\id[P]$ is at stake. By construction $e$ is a map of cones and hence $e^\dag$ is a map of cocones. Hence the bottom path in region (i) of
    \[\begin{tikzpicture}
     \matrix (m) [matrix of math nodes,row sep=2em,column sep=3.5em,minimum width=2em]
     {
      P & (A_1\oplus A_2,d_1,d_2) & (A_1\oplus A_2,i_1^\dag,i_2^\dag) & (A_1\oplus A_2,p_1,p_2) \\
       &  P  & (A_1\oplus A_2,d_1,d_2) \\
       && P \\};
     \path[->]
     (m-1-1) edge node [above] {$e$} (m-1-2)
             edge node [below] {$\id$} (m-2-2)
             edge[out=45,in=135,looseness=.5]  node [above] {$\tilde{g}$} (m-1-4)
     (m-1-2) edge node [above] {$f^\dag$} (m-1-3)
             edge node [left] {$e^\dag$} (m-2-2)
             edge node [above] {$h$} (m-2-3)
     (m-1-3) edge node [above] {$gg^\dag$} (m-1-4)
     (m-1-4) edge[out=-112.5,in=0,looseness=.5]  node[below] {$\tilde{f}$} (m-3-3) 
             edge node [above] {$f$} (m-2-3)
     (m-2-2) edge node [below] {$e$} (m-2-3)
             edge node [below] {$\id$} (m-3-3)
      (m-2-3) edge node [right] {$e^\dag$} (m-3-3);
      \node[gray] at (-0.6,2.5) {(i)};
      \node[gray] at (3.4,-0.15) {(ii)};
      \node[gray] at (1.1,.65) {(iii)};
    \end{tikzpicture}\]
 is a map of cones and thus equal to $\tilde{g}$, establishing commutativity of region (i). Similarly, $e^\dag f$ is a composite of two maps of cocones and hence equal to $\tilde{f}$, so that region (ii) commutes. Region (iii) commutes by definition of $h$ and the remaining triangles on the left commute by definition of $e$. Hence the whole diagram commutes, showing that $\tilde{f}\circ \tilde{g}=\id[P]$, as desired.
\end{proof}

\begin{corollary} 
  A dagger category admitting polar decomposition has finitely based dagger limits if and only if its underlying category has equalizers, intersections and 
  either (i) finite biproducts or (ii) finite products $(A_1\dots A_,p_1,\dots p_n)$ satisfying $p^\dag_i p_ip^\dag_j p_j=p^\dag_j p_jp^\dag_i p_i$. 
\end{corollary}
\begin{proof}
  Case (i) follows from Example~\ref{ex:polarlimitsaredaggerlimits} and Theorems~\ref{thm:finitecompleteness} and \ref{thm:polarofbiprod}, whereas case (ii) follows directly from Theorems~\ref{thm:finitecompleteness} and~\ref{thm:polarlimitsaredaggerlimits}.
\end{proof}

What made the previous results work for based diagrams is that we did not need to worry about the path taken to an object. However, if $A,B\in\Omega$ are distinct and admit maps $f\colon A\to C\leftarrow B\colon g$, the equation $D(f)l_A=l_C=D(g)l_B$ need not imply that $D(f)p_A=D(g)p_B$. Hence forgetting about the bimorphisms is not possible for arbitrary diagrams, and to make it work one has to change the diagram so that joint polar decomposition becomes available. The following theorem makes precise this idea that ``polar decomposition turns limits into dagger limits of isomorphic diagrams''. Recall that a category is balanced if all bimorphisms in it are isomorphisms.

\begin{theorem}\label{thm:daglimsofisomorphicdiagrams} 
  Consider a diagram  $D\colon \cat{J}\to\cat{C}$ in a balanced dagger category $\cat{C}$ that admits polar decomposition.
  Suppose that the diagram has a limit $l_A \colon L \to D(A)$ such that $l_A^\dag l_A l_B^\dag l_B=l_B^\dag l_Bl_A^\dag l_A$ for all $A$ and $B$ in some weakly initial $\Omega$.
  There is a diagram $E$ naturally isomorphic to $D$ such that $(E,\Omega)$ has a dagger limit.
\end{theorem}
\begin{proof}
  Pick a polar decomposition $l_A=p_Ai_A=j_Ap_A$ for each $A \in \cat{J}$. As \cat{C} is balanced, each $i_A$ and $j_A$ is an isomorphism, so we can define a new diagram $E\colon \cat{J}\to\cat{C}$ by $E(A)=D(A)$ on objects, and by $E(f)=j_B^{-1}D(f)j_A$ on morphisms $f\colon A\to B$. By construction $j_A \colon E(A) \to D(A)$ forms a natural isomorphism $E\Rightarrow D$. Moreover, since $l_A=j_Ap_A$, in fact $L$ is a limit of $E$ with a limiting cone $p_A \colon L \to E(A)$ of partial isometries.
  It remains to check independence. This is done as in the proof of Theorem~\ref{thm:polarlimitsaredaggerlimits}. By assumption and Lemma \ref{lem:PiOfPolarCommutes}, $p_A^\dag p_A$ commutes with $l_B^\dag l_B$ when $A,B \in \Omega$. Then, by another application of the lemma $p_A^\dag p_A$ commutes with $p_B^\dag p_B$. 
\end{proof}

\section{Commutativity of limits and colimits}\label{sec:commutativity}

In this section we investigate to what extent dagger limits commute with dagger (co)limits. To avoid multiply-indexed morphisms, in this section we do not name maps that are part of an obvious limiting cone (or daggers thereof).


Let us start by looking at whether dagger limits commute with dagger limits.
Assume that \cat{C} has all $(\cat{J},\Omega)$-shaped dagger limits and all $(\cat{K},\Psi)$-shaped dagger limits. Then, a bifunctor $D\colon\cat{J}\times\cat{K}\to \cat{C}$ induces functors $\cat{J} \to \cat{C}$ and $\cat{K} \to \cat{C}$ defined by $j\mapsto \dlim^\Psi_k D(j,k)$ and $k\mapsto \dlim^\Omega_j D(j,k)$. Since limits commute with limits, there exists a canonical isomorphism $\dlim^\Omega_j \dlim^\Psi_k D(j,k) \simeq \dlim^\Psi_k\dlim^\Omega_j D(j,k)$ between the two limits of $D$. In keeping with dagger category theory, we would like this canonical isomorphism to be unitary. Moreover, we would like both sides to be dagger limits of $(D,\Omega\times\Psi)$. We now prove that this holds whenever \cat{J} and \cat{K} are based and \cat{C} has zero morphisms. Later we will relax this to arbitrary \cat{J} and \cat{K} under a technical condition on the bifunctor $D$, that we conjecture is in fact not necessary.

\begin{theorem}\label{thm:limscommutewithlims} 
  Let $\Omega$ and $\Psi$ be bases for \cat{J} and \cat{K}. Assume that \cat{C} has zero morphisms, all $(\cat{J},\Omega)$-shaped dagger limits, all $(\cat{K},\Psi)$-shaped dagger limits, and let $D\colon \cat{J}\times\cat{K}\to\cat{C}$ be a bifunctor. Then $\dlim^\Psi_k\dlim^\Omega_j D(j,k)$ and $\dlim^\Omega_j\dlim^\Psi_k D(j,k)$ are both dagger limits of $(D,\Omega\times\Psi)$ and hence unitarily isomorphic.
\end{theorem}
\begin{proof} 
  By symmetry, it suffices to prove the claim for $\dlim^\Psi_k\dlim^\Omega_j D(j,k)$. Since ordinary limits commute with limits, $\dlim^\Psi_k\dlim^\Omega_j D(j,k)$ is a limit of $D$. Hence we need only check normalization and independence for $(j,k)\in\Omega\times\Psi$. In the presence of zero morphisms, we can use the simpler description from Proposition~\ref{prop:limswithzeros}.

  For normalization, we start by showing that, for fixed $(j,k)\in\Omega\times \Psi$, the morphism $p_k$ defined as the composition of canonical morphisms
  \[
    \dlim^\Omega_j D(j,k)\to\dlim^\Psi_k\dlim^\Omega_j D(j,k)\to\dlim^\Omega_j D(j,k)\to D(j,k)\to \dlim^\Omega_j D(j,k)
  \]
  factors through the canonical morphism $\dlim^\Psi_k\dlim^\Omega_j D(j,k)\to \dlim^\Omega_j D(j,k)$. This is done by extending it to a cone on $\dlim^\Omega_j D(j,k)$ for the functor $\dlim^\Omega_j D(j,-)\colon\cat{K}\to\cat{C}$. For $h$ in $\cat{K}$, define $p_h\colon \dlim^\Omega_j D(j,k)\to \dlim^\Omega_j D(j,h)$ by 
  \[
    p_h = \begin{cases}
      \dlim^\Omega_j D(j,f)p_k & \text{if }\cat{K}(k,h)\neq\emptyset\text, \\
      0 & \text{otherwise.}
    \end{cases}
  \]
  To see that this defines a cone, it suffices to check that $p_h$ is independent of the choice of $f\colon k\to h$. By the universal property of $\dlim^\Omega_j D(j,h)$, we may postcompose with a projection to an arbitrary $D(i,h)$ with $i\in\Omega$ and show that the resulting morphism does not depend on the choice of $f$. This splits into two cases depending on whether $i\neq j$ or not. If $i\neq j$, then the end result is always zero, since the following diagram commutes for every $f$:
  \[\begin{tikzpicture}
    \matrix (m) [matrix of math nodes,row sep=2em,column sep=6em,minimum width=2em]
    {
     &\dlim^\Omega_j D(j,k)& \dlim^\Omega_j D(j,h) \\
    D(j,k) & D(i,k) & D(i,h) \\
    & &  D(i,h) \\};
    \path[->]
    (m-1-2) edge node [left] {$$} (m-2-2)
           edge node [above] {$\dlim^\Omega_j D(j,f)$} (m-1-3)
    (m-1-3) edge node [right] {$$} (m-2-3)
    (m-2-2) edge node [below] {$D(i,f)$} (m-2-3)
    (m-2-3) edge node [right] {$\id$} (m-3-3)
    (m-2-1) edge[out=90,in=180,looseness=.5] node [above] {$$} (m-1-2)
            edge node [above] {$0$} (m-2-2)
            edge[out=-90,in=180,looseness=.4] node [below] {$0$} (m-3-3);       
  \end{tikzpicture}\]
  In case $i=j$, we will prove that the following diagram commutes:
  \[\begin{tikzpicture}
    \matrix (m) [matrix of math nodes,row sep=2em,column sep=2em,minimum width=2em]
    {
      \dlim^\Omega_j D(j,k) & \dlim^\Psi_k \dlim^\Omega_j D(j,k) & \dlim^\Omega_j D(j,k) & D(j,k)\\
     && D(j,k) & \dlim^\Omega_j D(j,k) \\
     & \dlim^\Omega_j D(j,h) &D(j,h)&  \dlim^\Omega_j D(j,h)\\};
    \path[->]
    (m-1-1) edge node [above] {$$} (m-1-2)
    (m-1-2) edge node [above] {$$} (m-1-3)
    (m-1-3) edge node [above] {$$} (m-1-4)
            edge[out=-160,in=90] node [left=2mm] {$\dlim^\Omega_j D(j,f)$} (m-3-2)
            edge node [left] {$$} (m-2-3)
    (m-1-4) edge node [above] {$$} (m-2-4)
    (m-3-2) edge node [below] {$$} (m-3-3)
    (m-2-3) edge node [right] {$D(j,f)$} (m-3-3)
    (m-2-4) edge node [below] {$$} (m-2-3)
             edge node [right] {$\dlim^\Omega_j D(j,f)$} (m-3-4)
    (m-3-4) edge node [below] {$$} (m-3-3);
  \end{tikzpicture}\]
  The top right square commutes because $j\in\Omega$, and the rest of the diagram commutes by definition of $ \dlim^\Omega_j D(j,f)$. The path along the top is $p_h$ followed by a projection to $D(j,h)$, whereas the other path is independent of the choice of $f\colon k\to h$ since $\dlim^\Psi_k \dlim^\Omega_j D(j,k)$ is a cone for $ \dlim^\Omega_j D(j,-)$. Thus \[p_h \colon \dlim^\Omega_j D(j,k) \to \dlim^\Omega_j D(j,h)\] forms a cone. Hence it factors through $\dlim^\Psi_k\dlim^\Omega_j D(j,k)$. In particular, $p_k$ factors through $\dlim^\Psi_k\dlim^\Omega_j D(j,k)\to \dlim^\Omega_j D(j,k)$. By Remark~\ref{rem:pisetc} this implies that the following diagram commutes: 
    \[\begin{tikzpicture}[font=\footnotesize]
       \matrix (m) [matrix of math nodes,row sep=2em,column sep=1.25em,minimum width=2em]
       {
        \dlim^\Omega_j D(j,k) & \dlim^\Psi_k\dlim^\Omega_j D(j,k) & \dlim^\Omega_j D(j,k) & D(j,k) & \dlim^\Omega_j D(j,k)\\
        \dlim^\Psi_k\dlim^\Omega_j D(j,k)& \dlim^\Omega_j D(j,k) & D(j,k) & \dlim^\Omega_j D(j,k) & \dlim^\Psi_k\dlim^\Omega_j D(j,k) \\};
       \path[->]
       (m-1-1) edge node [above] {$$} (m-1-2)
               edge node [left] {$$} (m-2-1)
       (m-1-2) edge node [above] {$$} (m-1-3)
       (m-1-3) edge node [above] {$$} (m-1-4)
       (m-1-4) edge node [above] {$$} (m-1-5)
       (m-2-1) edge node [below] {$$} (m-2-2)
       (m-2-2) edge node [below] {$$} (m-2-3)
       (m-2-3) edge node [below] {$$} (m-2-4)
      (m-2-4) edge node [below] {$$} (m-2-5)
      (m-2-5) edge node [right] {$$} (m-1-5);
    \end{tikzpicture}\]
    As the path along the bottom is self-adjoint, so is the top path. So the projections \[\dlim^\Omega_j D(j,k)\to \dlim^\Psi_k\dlim^\Omega_j D(j,k)\to \dlim^\Omega_j D(j,k)\] and \[\dlim^\Omega_j D(j,k)\to D(j,k)\to \dlim^\Omega_j D(j,k)\] commute. Finally, Remark~\ref{rem:pisetc}  guarantees that the composite $\dlim^\Psi_k\dlim^\Omega_j D(j,k)\to \dlim^\Omega_j D(j,k)\to D(j,k)$ is a partial isometry.

    For independence, it suffices to show for $(j,k),(i,h)\in\Omega\times\Psi$ that the composite $D(j,k)\to \dlim^\Omega_j D(j,k)\to \dlim^\Psi_k\dlim^\Omega_j D(j,k)\to \dlim^\Omega_j D(j,h)\to D(i,h)$ is zero when $(j,k)\neq (i,h)$. If $k\neq h$ this follows from the fact that $\dlim^\Omega_j D(j,k)\to \dlim^\Psi_k\dlim^\Omega_j D(j,k)\to \dlim^\Omega_j D(j,h)$ is zero. Hence we may assume that $h=k$ and consider the case $i\neq j$. As above, the projections  \[\dlim^\Omega_j D(j,k)\to \dlim^\Psi_k\dlim^\Omega_j D(j,k)\to \dlim^\Omega_j D(j,k)\] and \[\dlim^\Omega_j D(j,k)\to D(j,k)\to \dlim^\Omega_j D(j,k)\] commute. Hence the following diagram commutes:
    \[\begin{tikzpicture}
      \matrix (m) [matrix of math nodes,row sep=2em,column sep=3em,minimum width=2em]
      {
       D(j,k)& \dlim^\Omega_j D(j,k) & \dlim^\Psi_k\dlim^\Omega_j D(j,k) & \dlim^\Omega_j D(j,k) \\
       \dlim^\Omega_j D(j,k) & D(j,k)& D(j,k)& D(i,k) \\
       \dlim^\Psi_k\dlim^\Omega_j D(j,k) & \dlim^\Omega_j D(j,k) \\};
      \path[->]
      (m-1-1) edge node [left] {$$} (m-2-1)
             edge node [above] {$$} (m-1-2)
      (m-1-2) edge node [above] {$$} (m-1-3)
      (m-1-3) edge node [above] {$$} (m-1-4)
      (m-1-4) edge node [right] {$$} (m-2-4)
      (m-2-1) edge node [below] {$$} (m-2-2)
              edge node [left] {$$} (m-3-1)
      (m-2-2) edge node [right] {$$} (m-1-2)
      (m-2-3) edge node [below] {$0$} (m-2-4)
              edge node [above] {$$} (m-1-4)
      (m-3-1) edge node [below] {$$} (m-3-2)
      (m-3-2) edge node [below] {$$} (m-2-3);        
    \end{tikzpicture}\]
    This concludes the proof.
\end{proof}

Next, consider whether dagger limits commute with dagger colimits.
If \cat{C} has all $(\cat{J},\Omega)$-shaped dagger limits and all $(\cat{K},\Psi)$-shaped dagger \emph{co}limits (\ie $(\cat{K}\op,\Psi)$-shaped dagger limits), it is natural to ask when the canonical morphism 
\[\tau:\dcolim^\Psi_k\dlim^\Omega_j D(j,k)\to\dlim^\Omega_j\dcolim^\Psi_k D(j,k)\]
 is unitary. This canonical morphism $\tau$ and morphisms $\alpha_k$ are defined by making the following diagram commute for each $k$: 
\[\begin{tikzpicture}
  \matrix (m) [matrix of math nodes,row sep=2em,column sep=4em,minimum width=2em]
  {
   \dcolim^\Psi_k\dlim^\Omega_j D(j,k) &  \dlim^\Omega_j D(j,k)  & D(j,k) \\
   & \dlim^\Omega_j\dcolim^\Psi_k D(j,k)  & \dcolim^\Psi_k D(j,k) \\};
  \path[->]
  (m-1-1) edge[dashed] node [below] {$\tau$} (m-2-2)
  (m-1-2) edge node [above] {$$} (m-1-1)
  (m-1-2) edge node [above] {$$} (m-1-3)
  (m-1-3) edge node [right] {$$} (m-2-3)
  (m-1-2) edge[dashed] node [right] {$\alpha_k$} (m-2-2)
  (m-2-2) edge node [below] {$$} (m-2-3);
\end{tikzpicture}\]
A priori one might hope $\tau$ to be unitary very generally. After all, $(\cat{K},\Psi)$-shaped dagger colimits are just $(\cat{K}\op,\Psi)$-shaped dagger limits, so one might expect that commutativity of limits with colimits boils down to commutativity of limits with limits. To be slightly more precise, given $D\colon \cat{J}\times\cat{K}\to\cat{C}$, one would like to define $\hat{D}\colon \cat{J}\times\cat{K}\op\to\cat{C}$ by ``applying the dagger to the second variable'' and then calculating as follows:
\begin{align*}
  &\dcolim^\Psi_k\dlim^\Omega_j D(j,k)=\dlim^\Psi_k\dlim^\Omega_j \hat{D}(j,k) \\
  &\simeq_\dag \dlim^\Omega_j\dlim^\Psi_k \hat{D}(j,k)=\dlim^\Omega_j\dcolim^\Psi_k D(j,k)
\end{align*}
This, however, is a trap: $\hat{D}$ is not guaranteed to be a bifunctor. Indeed, the formula 
$
  \hat{D}(f,g)=D(f,\id)D(g,\id)^\dag
$
defines a bifunctor $\cat{J}\times\cat{K}\op\to\cat{C}$ if and only if every morphism $(f,g) \colon (j,k)\to (i,h)$ in $\cat{J}\times\cat{K}$ makes the following diagram commute:
\[\begin{tikzpicture}
  \matrix (m) [matrix of math nodes,row sep=2em,column sep=7em]
  {
   D(j,h)& D(i,k) \\
   D(j,k) & D(i,k) \\};
  \path[->]
  (m-1-1) edge node [left] {$D(\id,g)^\dag$} (m-2-1)
         edge node [above] {$D(f,\id)$} (m-1-2)
  (m-1-2) edge node [right] {$D(\id,g)^\dag$} (m-2-2)
  (m-2-1) edge node [below] {$D(f,\id) $} (m-2-2);       
\end{tikzpicture}\]

\begin{definition}
  A bifunctor $D \colon \cat{J} \times \cat{K} \to \cat{C}$ into a dagger category $\cat{C}$ is \emph{adjointable} when $D(-,g) \colon D(-,k) \Rightarrow D(-,h)$ is an adjointable natural transformation for each $g \colon k \to h$ in $\cat{K}$.
  Equivalently, $D$ is adjointable when $D(f,-)$ is an adjointable natural transformation for each morphism $f$ in $\cat{J}$.
\end{definition}

Let us temporarily go back to considering whether dagger limits commute with dagger limits. The extra condition of adjointability of $D$ lets us prove that this is true for arbitrary shapes $\cat{J}$ and $\cat{K}$. We conjecture that the extra condition is in fact not needed.

\begin{theorem}\label{thm:limscommutewithlims2} 
  Assume that \cat{C} has all $(\cat{J},\Omega)$-shaped dagger limits, all $(\cat{K},\Psi)$-shaped dagger limits, and let $D\colon \cat{J}\times\cat{K}\to\cat{C}$ be an adjointable bifunctor. Then $\dlim^\Psi_k\dlim^\Omega_j D(j,k)$ and $\dlim^\Omega_j\dlim^\Psi_k D(j,k)$ are both dagger limits of $(D,\Omega\times\Psi)$ and hence unitarily isomorphic.
\end{theorem}
\begin{proof} 
  By symmetry, it suffices to prove the claim for $\dlim^\Psi_k\dlim^\Omega_j D(j,k)$. Since ordinary limits commute with limits, $\dlim^\Psi_k\dlim^\Omega_j D(j,k)$ is the limit of $D$. Hence we only need to check normalization and independence for $(j,k)\in\Omega\times\Psi$. Consider the following diagram for some morphism $f \colon k \to h$ in $\cat{K}$:
  \[\begin{tikzpicture}
    \matrix (m) [matrix of math nodes,row sep=1em,column sep=6em,minimum width=2em]
    {
    \dlim^\Omega_j D(j,k)& \dlim^\Omega_j D(j,h) \\
     D(j,k) & D(j,h) \\
     \dlim^\Omega_j D(j,k)& \dlim^\Omega_j D(j,h)\\};
    \path[->]
    (m-1-1) edge node [left] {$$} (m-2-1)
           edge node [above] {$\dlim^\Omega_j D(j,f)$} (m-1-2)
    (m-1-2) edge node [right] {$$} (m-2-2)
    (m-2-1) edge node [below] {$D(\id,f)$} (m-2-2)
            edge node [left] {$$} (m-3-1)
     (m-2-2) edge node [right] {$$} (m-3-2)
     (m-3-1) edge node [below] {$\dlim^\Omega_j D(j,f)$} (m-3-2);       
  \end{tikzpicture}\]
  The top square commutes by definition of $\dlim^\Omega_j D(j,f)$. Since $D(-,f)$ is adjointable, the bottom square commutes too by Corollary~\ref{cor:mapoflimsismapofcolims}. Hence the whole diagram commutes, and the family $\sigma_k\colon \dlim^\Omega_j D(j,k)\to D(j,k)\to \dlim^\Omega_j D(j,k)$ is natural in $k$. Since $\sigma\colon \dlim^\Omega_j D(j,-)\to \dlim^\Omega_j D(j,-)$ is natural and pointwise self-adjoint, it is an adjointable natural transformation. Lemma~\ref{lem:connectingmaps} now makes the following diagram commute for each $j \in \Omega$ and $h,k \in \Psi$:
  \begin{equation*}\label{diag:commutinglimswithlims}\tag{$\star\star$}
    \begin{aligned}\begin{tikzpicture}
         \matrix (m) [matrix of math nodes,row sep=2em,column sep=1.5em,minimum width=2em]
         {
          \dlim^\Omega_j D(j,k) & D(j,k) & \dlim^\Omega_j D(j,k) & \dlim^\Psi_k\dlim^\Omega_j D(j,k)\\
          \dlim^\Psi_k\dlim^\Omega_j D(j,k)& \dlim^\Omega_j D(j,h) & D(j,h) & \dlim^\Omega_j D(j,h) \\};
         \path[->]
         (m-1-1) edge node [above] {$$} (m-1-2)
                 edge node [left] {$$} (m-2-1)
         (m-1-2) edge node [above] {$$} (m-1-3)
         (m-1-3) edge node [above] {$$} (m-1-4)
         (m-1-4) edge node [above] {$$} (m-2-4)
         (m-2-1) edge node [below] {$$} (m-2-2)
         (m-2-2) edge node [below] {$$} (m-2-3)
         (m-2-3) edge node [below] {$$} (m-2-4);
    \end{tikzpicture}\end{aligned}
  \end{equation*}
  Choosing $k=h$ shows that the projections $\dlim^\Omega_j D(j,k)\to D(j,k)\to \dlim^\Omega_j D(j,k)$ and $\dlim^\Omega_j D(j,k)\to \dlim^\Psi_k\dlim^\Omega_j D(j,k)\to \dlim^\Omega_j D(j,k)$ commute. Remark~\ref{rem:pisetc} then guarantees that the composite $\dlim^\Psi_k\dlim^\Omega_j D(j,k)\to \dlim^\Omega_j D(j,k)\to D(j,k)$ is a partial isometry for each $(j,k)\in\Omega\times\Psi$, establishing normalization.

  \begin{figure}\centering\begin{sideways}
      \begin{tikzpicture}[xscale=4.5, yscale=1.75, font=\footnotesize]
        \node (a1) at (1,9) {$\dlim^\Psi_k \dlim^\Omega_j D(j,k)$};
        \node (a2) at (2,9) {$\dlim^\Omega_j D(j,k)$};
        \node (a3) at (3,9) {$D(j,k)$};
        \node (a4) at (4,9) {$\dlim^\Omega_j D(j,k)$};
        \node (a5) at (5,9) {$\dlim^\Psi_k \dlim^\Omega_j D(j,k)$};
        \node (b1) at (1,8) {$\dlim^\Omega_j D(j,h)$};
        \node (b2) at (2.45,8) {$\dlim^\Psi_k \dlim^\Omega_j D(j,k)$};
        \node (b3) at (4,8) {$\dlim^\Omega_j D(j,h)$};
        \node (b4) at (4.5,8) {$D(j,h)$};
        \node (b5) at (5,8) {$\dlim^\Omega_j D(j,h)$};
        \node (c1) at (1.5,7) {$\dlim^\Psi_k \dlim^\Omega_j D(j,k)$};
        \node (c2) at (2.7,7.5) {$\dlim^\Omega_j D(j,h)$};
        \node (c3) at (3,7) {$\dlim^\Psi_k \dlim^\Omega_j D(j,k)$};
        \node (c4) at (4,7) {$D(i,h)$};
        \node (c5) at (5,7) {$D(i,h)$};
        \node (d1) at (1,6) {$D(i,h)$};
        \node (d2) at (2,6) {$\dlim^\Omega_j D(j,k)$};
        \node (d3) at (3,5) {$\dlim^\Psi_k \dlim^\Omega_j D(j,k)$};
        \node (d4) at (4,6) {$\dlim^\Omega_j D(j,h)$};
        \node (e1) at (1,5) {$\dlim^\Omega_j D(j,h)$};
        \node (e2) at (2,5) {$D(i,k)$};
        \node (e3) at (2.75,4) {$D(j,k)$};
        \node (e4) at (4.5,5.5) {$D(j,h)$};
        \node (e5) at (5,5) {$\dlim^\Omega_j D(j,h)$};
        \node (f2) at (2,4) {$\dlim^\Omega_j D(j,k)$};
        \node (f3) at (3.5,4) {$\dlim^\Omega_j D(j,k)$};
        \node (f4) at (4.35,4) {$\dlim^\Psi_k\dlim^\Omega_j D(j,k)$};
        \node (g1) at (1.35,4) {$\dlim^\Psi_k \dlim^\Omega_j D(j,k)$};
        \node (g2) at (2,3) {$\dlim^\Psi_k \dlim^\Omega_j D(j,k)$};
        \node (g3) at (3.5,3) {$\dlim^\Psi_k \dlim^\Omega_j D(j,k)$};
        \node (g4) at (4,3.5) {$\dlim^\Omega_j D(j,k)$};
        \node (h1) at (1.35,2) {$\dlim^\Omega_j D(j,h)$};
        \node (h2) at (2,2) {$\dlim^\Omega_j D(j,h)$};
        \node (h3) at (3.5,2) {$\dlim^\Omega_j D(j,h)$};
        \node (h4) at (4,1.5) {$\dlim^\Psi_k \dlim^\Omega_j D(j,k)$};
        \node (i1) at (1,1) {$\dlim^\Psi_k \dlim^\Omega_j D(j,k)$};
        \node (i2) at (2.5,1) {$\dlim^\Psi_k \dlim^\Omega_j D(j,k)$};
        \node (i3) at (2.75,2) {$D(j,h)$};
        \node (i4) at (4.5,0.75) {$\dlim^\Omega_j D(j,k)$};
        \node (j2) at (1.75,1) {$\dlim^\Omega_j D(j,k)$};
        \node (j3) at (3,0) {$\dlim^\Omega_j D(j,k)$};
        \node (j4) at (4,0) {$D(j,k)$};
        \node (j5) at (5,0) {$\dlim^\Psi_k \dlim^\Omega_j D(j,k)$};
        \draw[->] (a1) to (a2);
        \draw[->] (a2) to (a3);
        \draw[->] (a3) to (a4);
        \draw[->] (a4) to (a5);
        \draw[->] (a1) to (b1);
        \draw[->] (a2) to (b2);
        \draw[->] (a5) to (b5);
        \draw[->] (b1) to (c1);
        \draw[->] (b2) to (c2);
        \draw[->] (b2) to (b3);
        \draw[->] (b3) to (b4);
        \draw[->] (b3) to (c4);
        \draw[->] (b4) to (b5);
        \draw[->] (b5) to (c5);
        \draw[->] (b1) to (d1);
        \draw[->] (c1) to (d2);
        \draw[->] (c2) to (c3);
        \draw[->] (c3) to (b3);
        \draw[->] (c4) to (d4);
        \draw[->] (c5) to (e5);
        \draw[->] (d1) to (e1);
        \draw[->] (d2) to (e2);
        \draw[->] (d2) to (c3);
        \draw[->] (d3) to (d4);
        \draw[->] (d4) to (e4);
        \draw[->] (e4) to (e5);
        \draw[->] (e1) to (g1);
        \draw[->] (e2) to (f2);
        \draw[->] (e3) to (f3);
        \draw[->] (e1) to (i1);
        \draw[->] (f2) to (d3);
        \draw[->] (f2) to (e3);
        \draw[->] (f2) to (g2);
        \draw[->] (f3) to (f4);
        \draw[->] (f3) to (g3);
        \draw[->] (f4) to (e5);
        \draw[->] (e5) to (j5);
        \draw[->] (g1) to (f2);
        \draw[->] (g3) to (g4);
        \draw[->] (g4) to (f4);
        \draw[->] (g1) to (h1);
        \draw[->] (g2) to (h2);
        \draw[->] (g3) to (h3);
        \draw[->] (h3) to (h4);
        \draw[->] (h1) to (i1);
        \draw[->] (h2) to (i2);
        \draw[->] (h2) to (i3);
        \draw[->] (i3) to (h3);
        \draw[->] (h4) to (i4);
        \draw[->] (i1) to (j2);
        \draw[->] (j2) to (i2);
        \draw[->] (i2) to (j3);
        \draw[->] (j3) to (j4);
        \draw[->] (j4) to (i4);
        \draw[->] (i4) to (j5);
        \draw[->] (i1) to[out=-90,in=180] (j3);
        \draw[gray,draw=none] (b1) to node{(i)} (b2);
        \draw[gray,draw=none] (a2) to node{(ii)} (b5);
        \draw[gray,draw=none] (c2) to node{(iii)} (b3);
        \draw[gray,draw=none] (b3) to node{(iv)} (e5);
        \draw[gray,draw=none] (d4) to node{(v)} (f4);
        \draw[gray,draw=none] (c3) to node{(vi)} (d3);
        \draw[gray,draw=none] (d1) to node{(vii)} (e2);
        \draw[gray,draw=none] (e1) to node{(viii)} (h1);
        \draw[gray,draw=none] (g1) to node{(ix)} (g2);
        \draw[gray,draw=none] (g2) to node{(x)} (g3);
        \draw[gray,draw=none] (f3) to node{(xi)} (g4);
        \draw[gray,draw=none] (f4) to node{(xii)} (i4);
        \draw[gray,draw=none] (h4) to node{(xiii)} (j3);
        \draw[gray,draw=none] (j3) to node{(xiv)} (i1);
      \end{tikzpicture}\end{sideways}
    \caption{Proof of independence in Theorem~\ref{thm:limscommutewithlims2}.}
    \label{fig:independence}
  \end{figure}
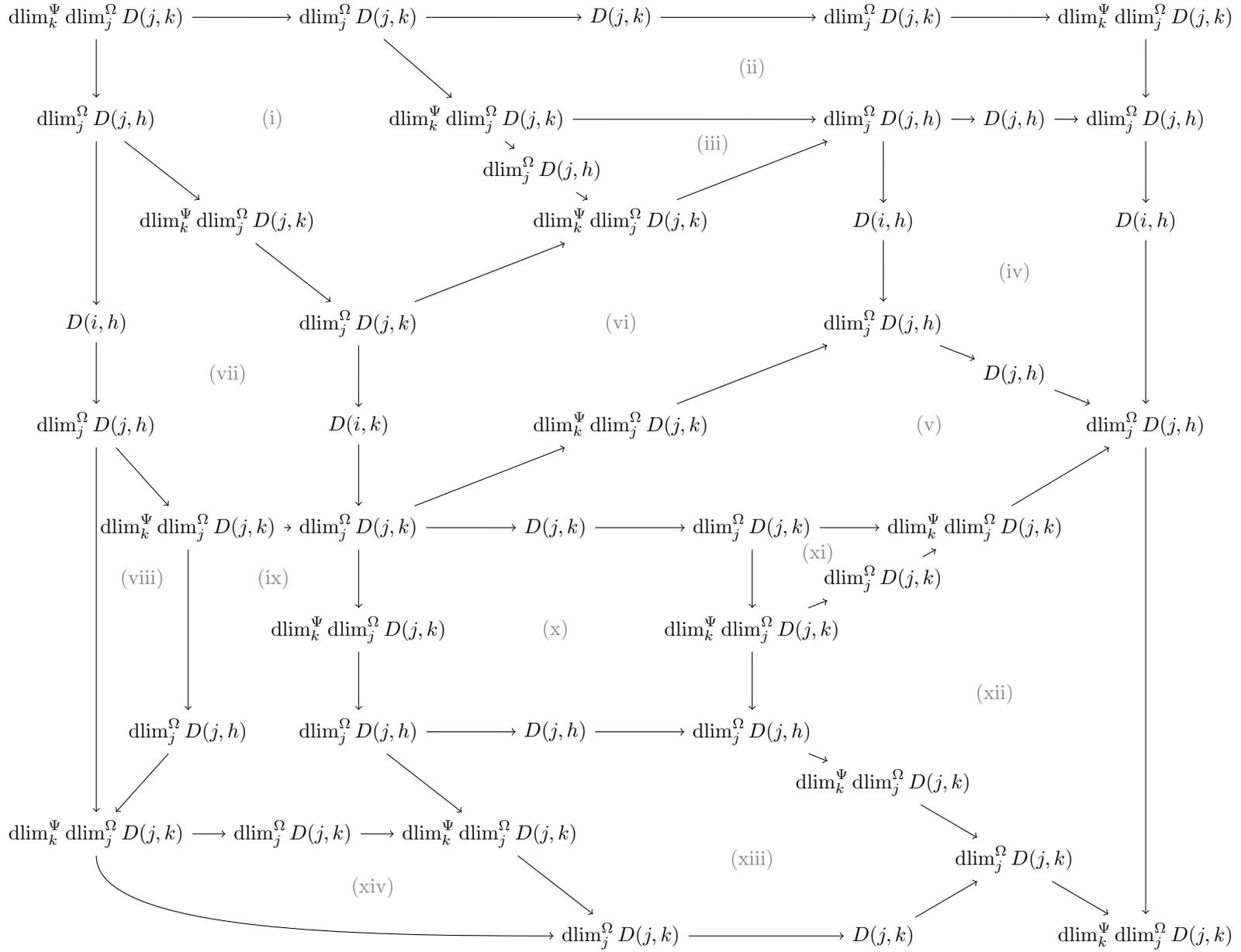

  For independence, pick $(j,k)$ and $(i,h)$ in  $\Omega\times\Psi$. We will show that the diagram in Figure~\ref{fig:independence} commutes. The fact that $k,h\in\Psi$ and $\dlim^\Psi_k\dlim^\Omega_j D(j,k)$ are normalized and independent at $\Psi$ ensures the commutativity of regions (i), (iii), (viii). (xi), (xi),(xii) and (xiv). Similarly, $i,j\in\Omega$ implies the commutativity of region (iv). The remaining regions, namely (ii), (v), (vi), (vii), (x) and (xiii), are all instances of diagram \eqref{diag:commutinglimswithlims}.
\end{proof}

We return to considering whether dagger limits commute with dagger colimits.
If $D$ is not adjointable, the colimit of limits need not be (unitarily) isomorphic to the limit of colimits.

\begin{example}\label{ex:limsneednotcommutewithcolims} 
  Dagger kernels need not commute with dagger cokernels in \cat{FHilb} if the bifunctor $D$ is not adjointable. For a counterexample, let $\cat{J}$ and $\cat{K}$ both be the shape $f,g \colon A \rightrightarrows B$ giving rise to equalizers.
  If $D\colon\cat{J}\times \cat{K}\to \cat{FHilb}$ maps each $D(f,-)$ and $D(-,f)$ to zero, then fixing the rest of $D$ corresponds to a choice of a commuting square in \cat{FHilb}. Let $D$ be thus defined by the square:
  \[\begin{tikzpicture}
    \matrix (m) [matrix of math nodes,row sep=2em,column sep=4em,minimum width=2em]
    {
     0 &\mathbb{C} &  \\
     \mathbb{C} &  \mathbb{C} \\};
    \path[->]
    (m-1-1) edge node [left] {$0$} (m-2-1)
           edge node [above] {$0$} (m-1-2)
    (m-1-2) edge node [right] {$\id $} (m-2-2)
    (m-2-1) edge node [below] {$\id $} (m-2-2);       
  \end{tikzpicture}\]
  Taking daggers of the horizontal arrows gives a square that does not commute, so $D$ is not adjointable. 
  Now, on the one hand, first taking cokernels horizontally and then taking kernels vertically gives $\mathbb{C}$: 
  \[\begin{tikzpicture}
    \matrix (m) [matrix of math nodes,row sep=2em,column sep=4em,minimum width=2em]
    { & & \ker 0=\mathbb{C} \\
     0 &\mathbb{C} & \coker 0=\mathbb{C} \\
     \mathbb{C} &  \mathbb{C} &\coker \id[\mathbb{C}]= 0\\};
    \path[->]
    (m-1-3) edge node [right] {$\id$} (m-2-3)
    (m-2-1) edge node [left] {$0$} (m-3-1)
           edge node [above] {$0$} (m-2-2)
    (m-2-2) edge node [right] {$\id $} (m-3-2)
            edge node [above] {$\id $} (m-2-3)
    (m-2-3) edge node [right] {$0$} (m-3-3)
    (m-3-1) edge node [below] {$\id $} (m-3-2)
    (m-3-2) edge node [below] {$0$} (m-3-3);       
  \end{tikzpicture}\]
  On the other hand, first taking kernels vertically and then taking cokernels horizontally gives $0$:
  \[\begin{tikzpicture}
         \matrix (m) [matrix of math nodes,row sep=2em,column sep=4em,minimum width=2em]
         {\ker 0=0 & \ker \id[\mathbb{C}]=0 &\coker 0=0 \\
          0 &\mathbb{C} \\
          \mathbb{C} &  \mathbb{C} \\};
         \path[->]
         (m-1-1) edge node [left] {$0$} (m-2-1)
                 edge node [above] {$0$} (m-1-2)
         (m-1-2) edge node [right] {$0$} (m-2-2)
                 edge node [above] {$0$} (m-1-3)
         (m-2-1) edge node [left] {$0$} (m-3-1)
                edge node [above] {$0$} (m-2-2)
         (m-2-2) edge node [right] {$\id $} (m-3-2)
         (m-3-1) edge node [below] {$\id $} (m-3-2);       
  \end{tikzpicture}\]
  Thus $\dlim_j\dcolim_k D(j,k)= \mathbb{C}$ is not isomorphic to $\dcolim_k \dlim_j D(j,k) = 0$, let alone unitarily so.
\end{example}

However, $D$ not being adjointable is the only obstruction to dagger limits commuting with dagger colimits.

\begin{theorem}\label{thm:limsandcolimscommute} 
  If \cat{C} has all $(\cat{J},\Omega)$-shaped dagger limits, all $(\cat{K},\Psi)$-shaped dagger colimits, and $D\colon \cat{J}\times\cat{K}\to\cat{C}$ is an adjointable bifunctor, then the canonical morphism 
  $\tau\colon\dcolim^\Psi_k\dlim^\Omega_j D(j,k)\to\dlim^\Omega_j\dcolim^\Psi_k D(j,k)$ is unitary. 
\end{theorem}
\begin{proof} 
  Because $D$ is adjointable, $\hat{D}(f,g)=D(f,\id)D(g,\id)^\dag$ defines an adjointable bifunctor $\hat{D} \colon \cat{J}\times\cat{K}\op\to\cat{C}$. It follows from Theorem~\ref{thm:limscommutewithlims2} that
  \begin{align*}
    \dcolim^\Psi_k\dlim^\Omega_j D(j,k)=\dlim^\Psi_k\dlim^\Omega_j \hat{D}(j,k) \\
    \simeq_\dag \dlim^\Omega_j\dlim^\Psi_k \hat{D}(j,k)=\dlim^\Omega_j\dcolim^\Psi_k D(j,k)
  \end{align*}
  Thus there exists a unitary morphism $u$ making the following diagram commute:
  \[\begin{tikzpicture}
     \matrix (m) [matrix of math nodes,row sep=2em,column sep=4em,minimum width=2em]
     {
      \dcolim^\Psi_k\dlim^\Omega_j D(j,k) &&  \dlim^\Omega_j D(j,k) \\
      \dlim^\Omega_j\dcolim^\Psi_k D(j,k) & \dcolim^\Psi_k D(j,k) & \dcolim^\Psi_k D(j,k) \\};
     \path[->]
     (m-1-1) edge[dashed] node [left] {$u$} (m-2-1)
            edge node [above] {$$} (m-1-3)
     (m-1-3) edge node [above] {$$} (m-2-3)
     (m-2-1) edge node [below] {$$} (m-2-2)
     (m-2-2) edge node [below] {$$} (m-2-3);
  \end{tikzpicture}\]
  We will show that $u$ satisfies the defining property of $\tau$, and hence equals it.
  Postcomposing with the morphisms $\dlim^\Omega_j\dcolim^\Psi_k D(j,k)\to \dcolim^\Psi_k D(j,k)\to D(j,k)$, this follows from commutativity of the following diagram:
  \[\begin{tikzpicture}
    \matrix (m) [matrix of math nodes,row sep=2.5em,column sep=4em,minimum width=2em]
    {
     \dlim^\Omega_j\dcolim^\Psi_k D(j,k) && \dcolim^\Psi_k D(j,k) \\
     \dcolim^\Psi_k\dlim^\Omega_j D(j,k) & \dlim^\Omega_j D(j,k) & D(j,k)\\
     \dlim^\Omega_j D(j,k) &D(j,k) & \dcolim^\Psi_k D(j,k) \\
     & \dlim^\Omega_j\dcolim^\Psi_k D(j,k) \\};
    \path[->]
    (m-1-1) edge node [above] {$$} (m-1-3)
    (m-1-3) edge node [right] {$$} (m-2-3)
     (m-2-1) edge node [left] {$u$} (m-1-1)
            edge node [below] {$$} (m-2-2)
     (m-2-2) edge node [below] {$$} (m-2-3)
     (m-3-1) edge node [above] {$$} (m-3-2)
           edge node [left] {$$} (m-2-1)
           edge node [below] {$\alpha_k$} (m-4-2)
      (m-3-2) edge node [above] {$$} (m-3-3)
      (m-3-3) edge node [right] {$$} (m-2-3)
      (m-4-2) edge node [below] {$$} (m-3-3);
  \end{tikzpicture}\]
  After all, the top part and bottom parts commute by definition of $u$ and $\alpha_k$. The commutativity of the middle rectangle is guaranteed by~Lemma \ref{lem:connectingmaps} once we show that the limiting cone morphisms $\dlim^\Omega_j D(j,k)\to D(j,k)$ are the components of an adjointable natural transformation $\dlim^\Omega_j D(j,-)\Rightarrow D(j,-)$. The naturality square commutes by definition of $\dlim^\Omega_j D(j,f)$:
  \[\begin{tikzpicture}
    \matrix (m) [matrix of math nodes,row sep=2em,column sep=6em,minimum width=2em]
    {
    \dlim^\Omega_j D(j,k)& \dlim^\Omega_j D(j,h) \\
     D(j,k) & D(j,h) \\};
    \path[->]
    (m-1-1) edge node [left] {$$} (m-2-1)
           edge node [above] {$\dlim^\Omega_j D(j,f)$} (m-1-2)
    (m-1-2) edge node [right] {$$} (m-2-2)
    (m-2-1) edge node [below] {$D(\id,f)$} (m-2-2);       
  \end{tikzpicture}\]
  Adjointability follows from Corollary~\ref{cor:mapoflimsismapofcolims} since $D(-,f)$ is adjointable:
  \[\begin{tikzpicture}
     \matrix (m) [matrix of math nodes,row sep=2em,column sep=6em,minimum width=2em]
     {
     \dlim^\Omega_j D(j,k)& \dlim^\Omega_j D(j,h) \\
      D(j,k) & D(j,h) \\};
     \path[->]
     (m-1-1) edge node [above] {$\dlim^\Omega_j D(j,f)$} (m-1-2)
     (m-2-2) edge node [right] {$$} (m-1-2)
     (m-2-1) edge node [below] {$D(\id,f)$} (m-2-2)
             edge node [left] {$$} (m-1-1);       
  \end{tikzpicture}\]
  This concludes the proof.
\end{proof}

\bibliographystyle{plain}
\bibliography{daggerlimits}

\end{document}